\documentclass[10pt]{amsart}

\usepackage[all]{xy}
\usepackage{fullpage}
\usepackage{latexsym}
\usepackage{amsmath}
\usepackage{amsfonts}
\usepackage{amssymb}
\usepackage{amsthm}
\usepackage{eucal}
\usepackage{enumerate,yfonts}
\usepackage{mathrsfs}
\usepackage{graphicx}
\usepackage{graphics}
\usepackage{epstopdf}
\usepackage{amscd}
\usepackage{bbm}
\usepackage{hyperref}
\usepackage{url}
\usepackage{color}
\usepackage{bbm}
\usepackage{cancel}
\usepackage{enumerate}
\usepackage{amsmath,amsthm}
\usepackage{amssymb}
\usepackage{epsfig}
\usepackage{pstricks}
\usepackage{xy}
\usepackage{xypic}

\newtheorem{thm}{Theorem}[section]
\newtheorem{corollary}[thm]{Corollary}
\newtheorem{lemma}[thm]{Lemma}
\newtheorem{lem}[thm]{Lemma}
\newtheorem{proposition}[thm]{Proposition}
\newtheorem{prop}[thm]{Proposition}

\newtheorem{thm-dfn}[thm]{Theorem-Definition}
\newtheorem{prop-dfn}[thm]{Proposition-Definition}

\theoremstyle{definition}
\newtheorem{definition}[thm]{Definition}
\newtheorem{remark}[thm]{Remark}
\newtheorem{example}[thm]{Example}

\numberwithin{equation}{section}

\newcommand{\fg}{{\mathfrak g}}
\newcommand{\ft}{{\mathfrak t}}

\newcommand{\fb}{{\mathfrak b}}

\newcommand{\fC}{{\mathfrak C}}

\newcommand{\fc}{{\mathfrak c}}

\newcommand{\rW}{{\mathrm W}}

\newcommand{\bC}{{\mathbb C}}

\newcommand{\bG}{{\mathbb G}}

\newcommand{\mD}{\mathcal{D}}

\newcommand{\mI}{\mathcal{I}}

\newcommand{\mE}{\mathcal{E}}
\newcommand{\mF}{\mathcal{F}}
\newcommand{\mQ}{\mathcal{Q}}
\newcommand{\mJ}{J}
\newcommand{\mA}{\mathcal{A}}
\newcommand{\mM}{\mathcal{M}}

\newcommand{\mO}{\mathcal{O}}
\newcommand{\mL}{\mathcal{L}}

\newcommand{\mG}{\mathcal{G}}
\newcommand{\mN}{\mathcal{N}}
\newcommand{\mB}{\mathcal{B}}

\newcommand{\sX}{\mathscr{X}}

\newcommand{\sG}{\mathscr{G}}

\newcommand{\sP}{\mathscr{P}}

\newcommand{\sH}{\mathscr{H}}

\newcommand{\sQ}{\mathscr{Q}}

\newcommand{\on}{\operatorname}

\newcommand{\ra}{\rightarrow}

\newcommand{\is}{\simeq}

\newcommand{\Loc}{\on{LocSys}}
\newcommand{\Bun}{\on{Bun}}

\newcommand{\tC}{\widetilde C}

\newcommand{\quash}[1]{}  
\newcommand{\nc}{\newcommand}

\newcommand{\frakb}{{\mathfrak b}}

\newcommand{\frakg}{{\mathfrak g}}

\newcommand{\frakn}{{\mathfrak n}}

\newcommand{\frakt}{{\mathfrak t}}

\newcommand{\bbA}{{\mathbb A}}

\newcommand{\bbC}{{\mathbb C}}

\newcommand{\bbG}{{\mathbb G}}

\newcommand{\bbV}{{\mathbb V}}

\newcommand{\calC}{{\mathcal C}}

\newcommand{\calF}{{\mathcal F}}
\newcommand{\calG}{{\mathcal G}}
\newcommand{\calH}{{\mathcal H}}
\newcommand{\calI}{{\mathcal I}}

\newcommand{\calO}{{\mathcal O}}

\newcommand{\calV}{{\mathcal V}}

\newcommand{\calZ}{{\mathcal Z}}

\nc{\al}{{\alpha}} \nc{\be}{{\beta}} \nc{\ga}{{\gamma}}
\nc{\ve}{{\varepsilon}} \nc{\Ga}{{\Gamma}} 
\nc{\La}{{\Lambda}}

\nc{\ad }{{\on{ad }}}

\nc{\aff}{{\on{aff}}} \nc{\Aff}{{\mathbf{Aff}}}
\newcommand{\Aut}{{\on{Aut}}}

\nc{\der}{{\on{der}}}
\newcommand{\Der}{{\on{Der}}}
\nc{\diag}{{\on{diag}}}
\newcommand{\End}{{\on{End}}}
\nc{\Fl}{{\calF\ell}}

\nc{\Hg}{{\on{Higgs}}}

\newcommand{\id}{{\on{id}}}
\nc{\Id}{{\on{Id}}}

\nc{\Ind}{{\on{Ind}}}
\newcommand{\Lie}{{\on{Lie}}}
\nc{\Op}{{\on{Op}}}
\newcommand{\Pic}{{\on{Pic}}}
\newcommand{\pr}{{\on{pr}}}

\nc{\res}{{\on{res}}}

\newcommand{\Spec}{{\on{Spec}}}

\nc{\tr}{{\on{tr}}}

\newcommand{\GL}{{\on{GL}}}
\nc{\GSp}{{\on{GSp}}} \nc{\GU}{{\on{GU}}} \nc{\SL}{{\on{SL}}}
\nc{\SU}{{\on{SU}}} \nc{\SO}{{\on{SO}}}

\nc{\nh}{{\Loc_{J^p}(\tau')}}
\nc{\bnh}{{\Loc_{\breve J^p}(\tau')}}

\nc{\bU}{{\overline{U}}} \nc{\IC}{{\on{IC}}}

\begin{document}
\title{Non-abelian Hodge theory for algebraic curves in characteristic $p$}
        \author{Tsao-Hsien Chen, Xinwen Zhu}
        \address{Department of Mathematics, Northwestern University,
        2033 Sheridan Road, Evanston, IL 60208, USA}
        \email{chenth@math.northwestern.edu}
        \address{Department of Mathematics, California Institute of Technology,
        1200 E. California Blvd., Pasadena, CA 91125, USA}
         \email{xzhu@caltech.edu}
\begin{abstract}
Let $G$ be a reductive group over an algebraically closed field of positive characteristic. Let $C$ be a smooth projective curve over $k$. We give a description of the moduli space of flat $G$-bundles in terms of the moduli space of $G$-Higgs bundles over the Frobenius twist $C'$ of $C$. This description can be regarded as the non-abelian Hodge theory for curves in positive characteristic. 
\end{abstract}
\maketitle

\setcounter{tocdepth}{1} \tableofcontents

\section{Introduction}
Let $C$ be a Riemann surface, and $G_c$ be a compact Lie group with $G$ its complexification. Consider the triple $(E,\nabla,\phi)$, where $E$ is a $C^\infty$-principal $G_c$-bundle, $\nabla$ is a $C^\infty$-connection on $C$, and $\phi\in \ad(E)\otimes\mA^{1,0}$ is a $(1,0)$ form on $C$ with values in the adjoint bundle $\ad(E)$. Then the famous Hitchin's equations are
\begin{equation}\label{Hitchin eqn}
\left\{\begin{split} F(\nabla)+[\phi,\phi^*]=0,\\ \nabla''(\phi)=0,
\end{split}
\right.
\end{equation}
where $F(\nabla)$ is the curvature of $\nabla$, $\phi^*$ is the complex conjugation of $\phi$ (with respect to $G_c$), and $\nabla''$ is the $(0,1)$-component of the connection $\nabla$. 

Hitchin obtained his equations by a dimensional reduction from $4D$ to $2D$ of the self-dual Yang-Mills equations, and found that the space of solutions of \eqref{Hitchin eqn} has spectacular geometric properties. Namely, let $( E,\nabla,\phi)$ be a solution. Then $\nabla''$ endows $E\otimes\bC$ with a structure as a holomorphic $G$-bundle, and $\phi$ as a holomorphic $\ad(E)$-valued 1-form, i.e $\phi\in \Gamma(C,\ad(E)\otimes\omega_C)$. Then the pair $(E,\phi)$ form a Higgs bundle. On the other hand, the connection $\tilde{\nabla}=\nabla+(\phi+\phi^*)$ is flat by \eqref{Hitchin eqn}. It is a well-known fact that a flat $C^\infty$-bundle on $C$ is the same as a holomorphic $G$-bundle with a flat connection, i.e., the pair $(E,\tilde\nabla)$ forms a de Rham $G$-local system. Hitchin showed that a Higgs bundle $(E,\phi)$ on $C$ arises as a solution of \eqref{Hitchin eqn} if and only it is (poly)stable with vanishing Chern class. Furthermore, the solution corresponding to $(E,\phi)$ is unique up to gauge transforms. In the other direction, Donaldson proved that every (semi)simple de Rham $G$-local system arises as a (unique up to gauge transforms) of the solution of \eqref{Hitchin eqn}. In this way, we get a bijection between stable $G$-Higgs bundles with vanishing Chern classes and irreducible $G$-local systems on $C$, which is furthermore a homeomorphism of the corresponding moduli spaces. This is what Simpson(\cite{S}) called the non-abelian Hodge theory for $C$\footnote{And what other people called the Simpson correspondence.}. Note that the $(0,1)$-form $\phi^*$ plays a fundamental role as its presence changes the holomorphic structure on $E$.

It is been a while to search a version of this correspondence in characteristic $p$ case (\cite{O,OV}). However, compared with the story over $\bC$, the picture is not very complete. In \cite{OV}, the Simpson correspondence is only established for connections with nilpotent $p$-curvatures.\quash{\footnote{But presumably this is the most interesting case since flat connections of geometric origin have nilpotent $p$-curvatures.}.}

\medskip

In this note, we will establish a full version of the non-abelian Hodge theory for curves. For this purpose, let us first try to write down an analogue of the Hitchin's equations (\ref{Hitchin eqn}) in characteristic $p$. Assume that $G$ and $C$ are defined over an algebraically closed field $k$ of characteristic $p>0$. We assume that $p$ does not divide the order of the Weyl group of $G$.
Let $F_C:C\to C'$ be the relative Frobenius map, where $C'$ is the pullback of $C$ along the absolute Frobenius of $k$. Let $(E,\tilde\nabla)$ be a de Rham $G$-local system on $C$. In characteristic $p$, one can associate to $(E,\tilde\nabla)$ its $p$-curvature, which is an $F$-Higgs field $\Psi(\tilde\nabla)\in \ad(E)\otimes F_C^*\omega_{C'}$. Let $\nabla^{can}$ denote the canonical connection on $F_C^*\omega_{C'}$. Let $\theta\in\ad(E)\otimes\omega_C$.
The analogue of Hitchin's equations \eqref{Hitchin eqn}  in characteristic $p$ we have in mind are
\begin{equation}\label{phe}
\left\{\begin{split} \Psi(\tilde\nabla-\theta)=0,\\ ((\tilde\nabla-\theta)\otimes\nabla^{can})(\Psi(\tilde\nabla))=0,
\end{split}
\right.
\end{equation}

Let us explain the meaning of these equations. First note that there are some formal analogy between \eqref{Hitchin eqn} and \eqref{phe}:  The first equation is about the ($p$-)curvature of a connection and the second equation is about the horizontality of some Lie algebra-valued one-from. The significance of these equations lies in the following observation. If we can find a solution $(E,\tilde\nabla,\theta)$ of \eqref{phe}, then we can endow $E$ with a new flat connection $\nabla=\tilde\nabla-\theta$, whose $p$-curvature vanishes by the first equation of \eqref{phe}. Therefore, by the Cartier decent, $E$ with this new connection is a pullback of a $G$-torsor $E'$ on $C'$. The second equation in \eqref{phe} implies that the $F$-Higgs field $\Psi(\tilde\nabla)$ is horizontal with respect to $\nabla$ and therefore is a pull back of $\phi'\in\ad(E')\otimes\omega_{C'}$, again by the Cartier decent. In other words, a solution of \eqref{phe} allows us to construct a Higgs bundle $(E',\phi')$ on $C'$ from a de Rham local system $(E,\tilde\nabla)$.  Note that the role $\theta$ is similar to the role of $\phi^*$ in \eqref{Hitchin eqn}. 

As is well-known, the $p$-curvature $\Psi(\tilde \nabla)$ is horizontal with respect to the connection $\tilde\nabla$ (cf. \cite[\S5]{K}). Therefore, the second equation in \eqref{phe} reduces to a matrix equation
\begin{equation}\label{2phe}[\Psi(\tilde\nabla),\theta]=0.\end{equation} 
On the other hand, the first equation \eqref{phe} is an ODE of order $p-1$. Indeed, it is well-known that if we choose a coordinate  $z$ on $C$, and let $A\in M_n(k[[z]])$ be a $k[[z]]$-valued $n\times n$ matrix, then the $p$-curvature of the connection $\partial_z+A$ can be represented as the matrix $(\partial_z+A)^{p-1}A$. Therefore, \eqref{phe} in principle should be much simpler than \eqref{Hitchin eqn} as the latter are non-linear PDEs.

Unfortunately, despite of its simple nature, we do not know how to find solutions of \eqref{phe} directly. The work of \cite[\S 4.1]{OV} essentially shows that if $G=\GL_n$, \eqref{phe} has solutions \'{e}tale locally on $C$. The proof is not elementary and is based on the Azumaya property of the sheaf of crystalline operators. It seems that there is some difficulty to generalize this approach to groups other than $\GL_n$. The main problem, as we shall see, is that there is no canonical solution of \eqref{phe}. So it is not clear (to us) how to apply the Tannakian formalism here.  In addition, the methods of \emph{loc. cit.} seems not to give the information of the existence of solutions of $\theta$ over $C$. In fact, as we shall argue, over the whole curve $C$,  one should consider another better-behaved system of equations, which has \eqref{phe} as its local avatar.

 \medskip

Note that \eqref{2phe} says that $\theta\in \Lie\!\ \Aut(E,\Psi(\tilde\nabla))$. The group scheme $\Aut(E,\Psi(\tilde\nabla))$ on $C$ is not well-behaved as the dimension of its fibers might jump. By the work of \cite{DG,N1}, the $F$-Higgs bundle $(E,\Psi(\tilde\nabla))$ (i.e. a Higgs bundle valued in $F^*\omega_{C'}$) defines a smooth commutative group scheme $J_{b^p}$ over $C$, where $b^p$ is the characteristic polynomial of $\Psi(\tilde \nabla)$. This is the so-called regular centralizer of the Higgs field $(E,\Psi(\tilde\nabla))$, which canonically maps to $\Aut(E,\Psi(\tilde\nabla))$. Our first observation is that instead of imposing $\theta\in\Lie \Aut(E,\Psi(\tilde\nabla))$, it is better to ask for $\theta\in\Lie\!\ J_{b^p}$. In other words, we replace the second equation of \eqref{phe} by asking for a section $\theta\in \Lie\!\ J_{b^p}$.
Our second observation is that 
the group scheme $J_{b^p}$ is the Frobenius pullback of a smooth group scheme $J'_{b'}$ on $C'$(cf. Proposition \ref{p-hitchin})\footnote{But $F$-Higgs bundle $(E,\Psi(\tilde\nabla))$ is not the Frobenius pullback!}, and it makes sense to talk about $J_{b^p}$-torsors with connections (cf. \S \ref{A}). In addition, given a $J_{b^p}$-torsor with a flat connection $(P,\nabla_{P})$, there is a ``product" flat connection $\nabla_{P\otimes  E}$ on the $G$-torsor $P\otimes  E:=P\times^{J_{b^p}} E$. Therefore, we may regard $\theta$ as a connection $d+\theta$ on $J_{b^p}$, and the first equation of \eqref{phe} says that the $p$-curvature of the ``product" connection on $E=J_{b^p}\otimes E$ vanishes. Finally, we replace $(J_{b^p}, d+\theta)$ by
 a general $J_{b^p}$-torsor with a connection $(P,\nabla_{P})$ and consider the equation 
\begin{equation}\label{globhe}
\Psi(\nabla_{P\otimes E})=0.
\end{equation}
As before, a solution $(E,\tilde\nabla,P,\nabla_P)$ of \eqref{globhe} defines a Higgs bundle $(E',\phi')$ on $C'$. Namely, since $\Psi(\tilde\nabla)$ is preserved by $J_{b^p}$, $(P\otimes E,\Psi(\tilde\nabla))$ is an $F$-Higgs bundle and \eqref{globhe}  implies that it is a pullback of some $(E',\phi')$. 

Note that such a quadruple $(E,\tilde\nabla,P,\nabla_P)$ is analogous to
a harmonic bundle as introduced by Simpson \cite{S}.
However, there is a fundamental difference. Namely, given $(E,\tilde\nabla)$, there might be more than one $(P,\nabla_P)$ making $(E,\tilde\nabla,P,\nabla_P)$ a solution of \eqref{globhe}. In fact, let $J'_{b'}$ be the smooth group scheme over $C'$ mentioned above such that $J_{b^p}=F^*J'_{b'}$, and let $\sP'_{b'}$ denote the Picard stack of $J'_{b'}$-torsors. Then the set of such $(P,\nabla_P)$, denoted by $\sH_{b'}^{-1}$, form a pseudo $\sP'_{b'}$-torsor\footnote{Recall that a pseudo torsor is either an empty set or a torsor.} by tensoring $(P,\nabla_P)$ with the Frobenius pullback of a $J'_{b'}$-torsor. By considering all possible characteristic polynomials, we have a Picard stack $\sP'$ over the Hitchin base $B'$ for the curve $C'$ and a pseudo $\sP'$-torsor $\sH^{-1}$ over $B'$. It is convenient to consider its inverse, denoted by $\sH$. In the main body of the paper, we will give another definition of 
$\sH$ as an algebraic stack over $B'$.
Our first result is
\begin{thm}\label{intro:torsor}
Under this action,
the stack $\sH$ form a $\sP'$-torsor.
\end{thm}
The essentially point is to show that the map $\sH\to B'$ is surjective.

Now we state the main theorem of the note. Let $\Hg'_G$ denote the moduli space of Higgs bundles on $C'$. According to \cite{N1,N2}, there is a canonical action of $\sP'$ on $\Hg'_G$. 
Let $ \sH\times^{\sP'}\Hg'_{G}$ be the twist of $\Hg'_{G}$ by the 
$\sP'$-torsor $\sH$.

\begin{thm}\label{intro:nah}
Over $B'$, there is a canonical isomorphism
\[\mathfrak C: \sH\times_{}^{\sP'}\Hg'_{G}\to \Loc_{G}.\]
\end{thm}
Roughly speaking, the map $\mathfrak C$ is inverse to the construction of a Higgs bundle from a solution of  \eqref{globhe} as mentioned above.

\begin{corollary}\label{et loc isom}
There is an \'{e}tale cover $U$ of $B'$ such that
$\Loc_{ G}\times_{B'}U$ is isomorphic to $\Hg'_{ G}\times_{B'}U$. 
\end{corollary}
\begin{proof}
The Picard stack $\sP'$ is smooth over $B'$, and therefore $\sH$ can be trivialized \'{e}tale locally on $B'$. 
\end{proof}
\begin{remark}
In the case
$G=\GL_n$, this corollary is one of the main theorems of \cite{G}, which extends a result of \cite{BB} that establishes the above isomorphism away from the discriminant locus of $B'$. However, the isomorphism $\fC$ was not obtained in \emph{loc. cit.}.
\quash{
(ii) In fact, $\sH$ can be trivialized along an ``Artin-Schneier" type cover $B'\to B'$. See \S \ref{trivialization} for details.}
\end{remark}

Theorem \ref{intro:nah} and Corollary \ref{et loc isom} can be regarded as a non-abelian Hodge theory in characteristic $p$ as it says that $\Hg'_G$ and $\Loc_G$ are isomorphic \'etale locally over $B'$. Let us discuss some new features. First, our theorem is stated for \emph{all} $G$-Higgs fields and  \emph{all} flat $G$-bundles, while over $\bC$ such a correspondence could only be possible for (poly)stable Higgs bundles with vanishing Chern classes and (semi)simple flat $G$-local systems. Second, in characteristic $p$, what exists is the isomorphism $\fC$ rather than a direct isomorphism between $\Hg'_{G}$ and $\Loc_{G}$. This is reminiscent of the transcendental nature of the Simpson correspondence over $\bC$. Third, as the isomorphism $\mathfrak C$ is algebraic, one can compare directly between the tangent spaces of $\Hg'_{G}$ and $\Loc_{G}$. In particular, one can obtain the comparison between Higgs cohomology and de Rham cohomology as a consequence of the global isomorphism $\fC$. 

The relation between our work and the construction of 
\cite{OV} (in the curve case) is as follows. The construction of \cite{OV} amounts to the restriction of $\fC$ to $0\in B'$. Namely, upon a choice of a lifting of $C'$ to $W_2(k)$, $\sH|_0$ admits a canonical trivialization (see \S \ref{triv}), and thus $\fC$ induces an isomorphism between the moduli of nilpotent $G$-Higgs bundles over $C'$ and the moduli of flat $G$-bundles with nilpotent $p$-curvatures. 
We remark that it will be clear from our construction that if the group $G$ is bigger than $\on{SL}_2$ there exist trivializations of $\sH|_0$ not induced from the lifting of $C'$ to $W_2(k)$ (See Remark \ref{other lifting}).

We also mention that the results in this note will play a fundamental role in our subsequent paper \cite{CZ} to establish the (generic) geometric Langlands duality in prime characteristic. The main step is to identify the $\sP'$-torsor $\sH$ with another $\sP'$-torsor over $B'$ constructed from the Langlands dual group. We refer \cite{CZ} for details.

Finally,\quash{we remark it is not very difficult to generalize our results to the parabolic version. But} it will be very interesting to find a generalization of these results to higher dimensional varieties. We just remark that Equation \eqref{phe} makes sense for higher dimensional varieties. But solutions should only exist for those varieties that can be lifted to $W_2(k)$ .

\medskip

Let us briefly summarize the following sections. We will review the theory of Hitchin fibrations in \S \ref{Hf}, mostly following \cite{N1,N2}. Readers familiar with this theory can skip most of it and go to \S \ref{tau and c} directly. Our main theorem is proved in \S \ref{pHit}. We construct the $p$-Hitchin map in \S \ref{p-Hitchin map}, and then prove Theorem \ref{intro:torsor} in \S \ref{proof} by some cohomological argument. We briefly discuss the trivialization of the $\sP'_0$-torsor $\sH|_0$ is in \S \ref{triv}. We establish our main theorem in \S \ref{NC}. The proof is rather formal, after we develop some theories in
Appendix \ref{A}, which might be of independent interest. More precisely, we first develop the theory of de Rham $\mG$-local systems for a non-constant group scheme $\mG$ over $X$. The main observation is that the only additional input is a flat connection on $\mG$ itself compatible with the group structure. Next, we discuss the $p$-curvature of such a de Rham $\mG$-local system. We then develop the notion of the scheme of horizontal sections of a $\mD_X$-scheme, which is the right adjoint of the pullback of $X'$-schemes along the relative Frobenius. This is analogous to the same named notion in characteristic zero case developed in (\cite[\S 2.6]{BD}).  We finish the note with a construction for every smooth commutative group scheme $\calG'$ over $X'$, a $4$-term exact sequence of \'{e}tale sheaves on $X'$, which for $\mG'=\bG_m$ specializes to a well-known exact sequence as in \cite[\S 2.1]{I}.

\medskip

We introduce the notations used throughout the note. We fix a smooth projective curve $C$ of genus at least two\footnote{This assumption should not be essential. We impose it to avoid the DG structure on moduli spaces.} over an algebraically closed field $k$ of characteristic $p>0$. Let $\omega=\omega_C$ denote the canonical line bundle of $C$. 

Let $S$ be a
Noetherian scheme and $\sX\ra S$ be an algebraic stack over $S$. If
$p\calO_S=0$, we denote by  $Fr_S:S\ra S$ be the absolute Frobenius
map of $S$. We have the following commutative diagram
$$\xymatrix{\sX\ar[r]^{F_{\sX/S}}\ar[dr]&\sX^{(S)}\ar^{\pi_{\sX/S}}[r]\ar[d]&\sX\ar[d]\\
&S\ar[r]^{Fr_S}&S}$$ where the square is Cartesian. We call
$\sX^{(S)}$ the Frobenius twist of $\sX$ along $S$, and
$F_{\sX/S}:\sX\ra\sX^{(S)}$ the relative Frobenius morphism. If the
base scheme $S$ is clear, $\sX^{(S)}$ is also denoted by $\sX'$ for
simplicity and $F_{\sX/S}$ is denoted by $F_{\sX}$ or $F$.

Let $\mG$ be a smooth
affine group scheme over $X$, and $E$ be a $\mG$-torsor on $X$. We
denote by $\Aut (E)=E\times ^{\mG}\mG$ the adjoint torsor and
$\ad(E)$ the adjoint bundle.

Let $G$ be a reductive algebraic group over $k$ of rank $l$. We denote by
$\fg$ the Lie algebra of $G$. We assume that $p$ does not divide the order of the Weyl group $\rW$ of $G$.

\subsection*{Acknowledgements}
T-H. Chen would like to thank Roman Bezrukavnikov for many helpful discussions. We thank M. Groechenig for pointing out a mistake in an earlier version of the note.
T-H. Chen is supported by NSF under the agreement No.DMS-1128155. X. Zhu is partially supported by NSF grant DMS-1001280/1313894 and DMS-1303296/1535464 and AMS Centennial Fellowship.

\section{Recollections of Hitchin fibrations} \label{Hf}
In this section, we review some basic geometric facts of Hitchin
fibrations, mostly following \cite{N1,N2}. Readers familiar with this theory can directly jump to the next section, probably except \S \ref{tau and c}.

\subsection{The Hitchin fibration}
Let $k[\fg]$ and $k[\ft]$ be the algebra of polynomial function on
$\fg$ and $\ft$. By Chevalley's theorem, we have an isomorphism
$k[\fg]^{G}\is k[\ft]^{\rW}$. Moreover, $k[\ft]^\rW$ is isomorphic to
a polynomial ring of $l$ variables $u_1,\ldots,u_l$ and each $u_i$ is
homogeneous in degree $e_i$. Let $\fc=\on{Spec}(k[\ft]^W)$. Let  $$\chi:\fg\ra\fc$$ be the map induced by $k[\fc]\is
k[\fg]^G\hookrightarrow k[\fg]$.
This is $G\times\bG_m$-equivariant
map where $G$ acts trivially on $\fc$, and $\bG_m$ acts on $\fc$ through the gradings on $k[\ft]^{\rW}$. Let $\mL$ be an invertible
sheaf on $C$ and $\mL^\times$ be the corresponding
$\bG_m$-torsor. Let
$\fg_{\mL}=\fg\times^{\bG_{m}}\mL^\times$ and
$\fc_{\mL}=\fc\times^{\bG_m}\mL^\times$ be  the $\bG_m$-twist of $\fg$
and $\fc$ with respect to the natural $\bG_m$-action.

Let $\on{Higgs}_{G,\mL}=\on{Sect}(C,[\fg_{\mL}/G])$ be the stack of
section of $[\fg_{\mL}/G]$ over $C$, i.e., for each $k$-scheme $S$
the groupoid  $\on{Higgs}_{G,\mL}(S)$ consists of maps over $C$:
\[h_{E,\phi}:C\times S\ra[\fg_{\mL}/G].\]
Equivalently, $\on{Higgs}_{G,\mL}(S)$ consists of a pair $(E,\phi)$ (called a Higgs bundle), where
$E$ is a $G$-torsor over $C\times S$ and $\phi$ is an element in
$\Gamma(C\times S,\ad (E)\otimes \mL)$. If the group $G$ is clear
from the content, we simply write $\on{Higgs}_{\mL}$ for
$\on{Higgs}_{G,\mL}$.

Let $B_{\mL}=\on{Sect}(C,\fc_{\mL})$ be the scheme of
sections of $\fc_{\mL}$ over $C$, i.e., for each $k$-scheme $S$,
$B_{\mL}(S)$ is the set of sections over $C$
$$b:C\times S\ra\fc_{\mL}.$$
This is called the Hitchin base of $G$.

The natural $G$-invariant projection $\chi:\fg\ra\fc$  induces a map
$$[\chi_{\mL}]:[\fg_{\mL}/G]\ra\fc_{\mL},$$
which in turn induces a natural map
$$h_{\mL}:\on{Higgs}_{\mL}=\on{Sect}(C,[\fg_{\mL}/G])\ra\on{Sect}(C,\fc_{\mL})=B_{\mL}.$$
We call $h_{\mL}:\on{Higgs}_{\mL}\ra B_{\mL}$ the Hitchin map
associated to $\mL$.
For any
$b\in B_\mL(S)$ we denote by $\on{Higgs}_{\mL,b}$ the fiber product
$S\times_{B_\mL}\on{Higgs}_{\mL}$.

We are mostly interested in the case $\mL=\omega$. For
simplicity, from now on we denote $B=B_{\omega}$,
$\on{Higgs}=\on{Higgs}_{\omega}$, $h=h_{\omega}$, and $\Hg_b=\Hg_{\omega_C,b}$, etc. We sometimes also write
$\on{Higgs}_G$ for $\on{Higgs}$ to emphasize the group $G$. 

We fix a
square root $\kappa=\omega^{1/2}$ (called a theta characteristic of
$C$). Recall that in this case, there is a section
$\epsilon_{\kappa}:B\ra\on{Higgs}$ of $h:\Hg\to B$, induced by the Kostant section $kos:\fc\ra\fg$. Sometimes, we also call $\epsilon_{\kappa}$ the Kostant section of the Hitchin fibration, and denote it by $\kappa$ for simplicity.

\subsection{Symmetries of Hitchin fibration}\label{symmetry of Hitchin}
Consider the group scheme $I$ over $\fg$ consisting of pairs
$$I=\{(g,x)\in G\times\fg\mid \on{Ad}_g(x)=x\}.$$
The group scheme $I$ is not flat, but when restricted to the open subset of regular elements $\fg^{reg}$, it is smooth.
We define $J=kos^{*}I$, This is called the universal centralizer group scheme of
$\fg$, which is a smooth commutative group scheme over $\fc$.
The following proposition is proved in \cite{DG, N1} 
\begin{proposition}\label{J}
There is a canonical isomorphism of group schemes
$\chi^{*}J|_{\fg^{reg}}\is I|_{\fg^{reg}}$, which extends to a
morphism of group schemes $a:\chi^{*}J\ra I\subset G\times\fg$.
\end{proposition}

All the above constructions can be twisted. Namely, there is a
$\bG_m$-action on $I$ given by $t\cdot(g,x)=(g,tx)$. It induced a 
$\bG_m$-action on $J=kos^*I$, and the natural morphisms $J
\ra\fc$ and $I\ra \fg$ are $\bG_m$-equivariant. Therefore 
we can twist everything by a $\bG_m$-torsor $\mL^\times$ and get 
$J_{\mL}\ra\fc_{\mL}$, $I_{\mL}\ra\fg_{\mL}$, where
$J_{\mL}=J\times^{\bG_m}\mL^\times$ and
$I_{\mL}=I\times^{\bG_m}\mL^\times$. 
Moreover, there is a $G$-action on $I$
given by $h\cdot(g,x)=(hgh^{-1},hxh^{-1})$. The group scheme $I_{\mL}\to \fg_{\mL}$ is $G$-equivariant, and hence descends to a group scheme $[I_{\mL}]$ over
$[\fg_{\mL}/G]$. 

There is a natural action of the Picard stack $\sP_\mL$ of 
$J_\mL$-torsors on $\on{Higgs}_{\mL}$. We recall its definition.
Let $b:S\ra B_{\mL}$ be an $S$-point of $B_{\mL}$, corresponding to a map
$b:C\times S\ra\fc_{\mL}$. Pulling back $J_\mL\ra\fc_{\mL}$
along this map defines a smooth group scheme $J_b=b^{*}J_\mL$ over $C\times
S$. Let $(E,\phi)\in\on{Higgs}_{\mL,b}$, corresponding to a map $h_{E,\phi}:C\times S\ra [\fg_{\mL}/G]$. Observe that the morphism $\chi^*J\to I$ in
Proposition \ref{J} induces $[\chi_{\mL}]^{*}J_\mL\ra [I_\mL]$ of group
schemes over $[\fg_{\mL}/G]$.
Pulling back to $C\times S$ by $h_{E,\phi}$, we get a map
\begin{equation}\label{actEphi}
a_{E,\phi}:J_b\ra h_{E,\phi}^{*}[I_\mL]=\Aut(E,\phi)\subset \Aut(E).
\end{equation}
Therefore, we can twist
$(E,\phi)\in\on{Higgs}_{\mL,b}$ by a $J_b$-torsor.  This defines the promised action
of $\sP_{\mL}$ on $\on{Higgs}_{\mL}$ over $B_{\mL}$. 

\subsection{The tautological section $\tau:\fc\to \Lie\!\ J$}\label{tau and c}
Recall that by
Proposition \ref{J}, there is a canonical isomorphism
$\chi^*J|_{\fg^{reg}}\simeq I|_{\fg^{reg}}$. The sheaf of Lie
algebras $\Lie\!\ (I|_{\fg^{reg}})\subset \fg^{reg}\times \fg$ admits a
tautological section given by $x\mapsto x\in\Lie\!\ I_x$ for $x\in
\fg^{reg}$. Clearly, this section descends to give a tautological
section $\tau: \fc\to \Lie\!\ J$. We have the following property of $\tau$.

\begin{lem}\label{taut:prolong}
Let $x\in \fg$, and $a_x: J_{\chi(x)}\to I_x\subset G$ be the homomorphism as in Proposition \ref{J} (1). Then $da_x(\tau\chi(x))=x$.
\end{lem}
\begin{proof}Consider the universal situation $x=\id:\fg\to\fg$. Then we need to show that $\chi^*(da\circ\tau):\fg\to  \chi^*\Lie\!\ J\to \fg\times \fg$ is the diagonal map. But by definition, this is true when restricted to $\fg^{reg}\subset \fg$. Therefore it holds over $\fg$. 
\end{proof}
\begin{remark}In particular, if we take $x=0$, the lemma shows that $da_0:\Lie J_{\chi(0)}\to \frakg$ is not injective. I.e. the map $a:\chi^*J\to I$ is neither injective nor surjective over a general point $x\in \fg$.
\end{remark}

Observe that $\bG_m$ acts on $\fg^{reg}\times \fg$ via natural
homothethies on both factors, and therefore on $\chi^*\Lie\!\
J|_{\fg^{reg}}\simeq \Lie\!\ (I|_{\fg^{reg}})\subset \fg^{reg}\times
\fg$. This $\bG_m$-action on $\chi^*\Lie\!\ J|_{\fg^{reg}}$ descends to
a $\bG_m$-action on $\Lie\!\ J$ and for any line bundle $\mL$ on $C$,
the $\mL^\times$-twist $(\Lie\!\ J)\times^{\bG_m}\mL^\times$ under this
$\bG_m$ action is $\Lie\!\ (J_\mL)\otimes \mL$, where $J_\mL$ is
introduced in \S \ref{symmetry of Hitchin}. In addition, $\tau$ is $\bG_m$-equivariant with respect to this $\bG_m$ action on
$\Lie\!\ J$ and the natural $\bG_m$ action on $\fc$. Therefore, if we define a vector bundle
$B_{J,\mL}$ over $B_{\mL}$, whose fiber over $b\in B_{\mL}$ is
$\Gamma(C,\Lie\!\ J_b\otimes \mL)$, then
twisting $\tau$ by $\mL$, we obtain 
\begin{equation}\label{taut sect II}
\tau_{\mL}: B_{\mL}\to B_{J,\mL}.
\end{equation}
which is a canonical section of the  projection $\pr:B_{J,\mL}\to
B_{\mL}$.


\section{The non-abelian Hodge theory}\label{pHit}
Let $\Loc_{G}$ (or $\Loc$ for brevity) be the moduli space of $G$-local systems on the curve
$C$. The natural map $\Loc_{G}\to \Bun_G$ can be regarded as a
deformation of the map $\pi:T^*\Bun_{G}\to \Bun_{G}$.  On the other
hand, $T^*\Bun_G$ is just the moduli of Higgs field $\on{Higgs}_G$
and there is the Hitchin fibration $h:\on{Higgs}_G\to B$. In this
section, we show that in the case $\on{char} k=p>0$, there is the $p$-Hitchin map
$h_p:\Loc_{G}\to B'$,
which is reminiscent of the classical
Hitchin map in many aspects. In particular, we prove the non-abelian Hodge theory for $C$, which among other things, implies that \'{e}tale locally on $B'$, $h_p$ is isomorphic to $h':\Hg'_G\to B'$.
For general discussions of local
systems on a smooth variety, we refer to Appendix \ref{A}.

\subsection{The $p$-Hitchin map for $G$-local systems}\label{p-Hitchin map}
Let $\Loc_{G}$ be the stack of $G$-local systems on $C$, i.e. for
every scheme $S$ over $k$, $\Loc_{G}(S)$ is the groupoid of all
$G$-torsors $E$ on $C\times S$ together with a connection $\nabla:
T_{C\times S/S}\ra\widetilde{T}_E$ (cf. \S \ref{Lie algebroid}). For every $(E,\nabla)\in\Loc_{G}$
the $p$-curvature  of $\nabla$ is an element
$\Psi(\nabla)\in\Gamma(C,\ad (E)\otimes F_C^*\omega')$ (cf. \S \ref{appen:p-curv}). 
Note that $F_X^*\omega'\simeq \omega^p$, the $p$-th power of $\omega$.
Thus we get an element $(E,\Psi(\nabla))$
in $\on{Higgs}_{G,\omega^p}$, the stack  of $F$-Higgs bundles, and the
assignment $(E,\nabla)\ra (E,\Psi(\nabla))$ defines a map
$$\Psi_{G}:\Loc_{ G}\ra\on{Higgs}_{G,\omega^p}.$$ 
Observe that pullback along $F_C:C\to C'$ induces a natural map
\[
F^p: B'\ra B_{\omega^p}.
\]
In what follows, for $b'\in B'$, we denote by $b^p$ the image of
$b'$ under the map $F^p:B'\ra B_{\omega^p}$.

\begin{prop}\label{p-hitchin}
There is a unique morphism $h_p:\Loc_{ G}\ra B'$, making the following diagram commute
\[\begin{CD}
\Loc_G@>h_p>> B'\\
@V\Psi_G VV @VV F^p V\\
\Hg_{G,\omega^p}@>h_{\omega^p}>> B_{\omega_p}.
\end{CD}\]
\end{prop}

\begin{proof}
The argument is a variation of the proof of \cite[Proposition
3.2]{LP} which is due to J.B. Bost. Let $(E,\nabla)\in\Loc_{G}$ and let $\Psi\in\Gamma(C,\ad
(E)\otimes\omega^p)$ be the $p$-curvature of $\nabla$. As $\fc_{\omega^p}$ is the pullback of $\fc_{\omega}$ along $F_C:C\to C'$ , it has a canonical connection $\nabla^{can}$ and
 the scheme of horizontal sections $\fc_{\omega^p}^{\nabla^{can}}$ of
$\fc_{\omega^p}$ is $\fc'_{\omega'}$ (see \S \ref{horizontal sect}). Therefore,
it is enough to show that the section $h_{\omega^p}(\Psi)\in\Gamma(C,\fc_{\omega^p})$
is horizontal with respect to $\nabla^{can}$. This follows from the next lemma and the fact that 
the $p$-curvature $\Psi$ is horizontal with respect to the tensor connection on $\ad(E)\otimes\omega^p$.
\quash{
This is a local question, therefore, we can assume
that $C=\on{Spec}(k[[z]])$ and that $\omega$, $E$ are trivialized. Then as sections of $\tilde T_E$ (the Lie algeboid of infinitesimal symmetries of $E$, cf. \S \ref{Lie algebroid}),
$\nabla^{can}=\partial_z$ and $\nabla=\partial_z-A$,
where $A\in \fg[[z]]$. In addition, the $p$-curvature 
$\Psi\in \ad E\otimes\omega^p\is\fg[[z]]$.
Since $\Psi$ is horizontal with respect to the induced connection on $\ad E\otimes \omega^p$, we have
\begin{equation}\label{p}
\partial_z\Psi=[A,\Psi],
\end{equation}
where $[\ , \ ]$ is the natural bracket on $\frakg[[z]]$. For a polynomial function $f$ on $\fg$, let $A\cdot f$ denote the action of $A\in \fg$ induced by the adjoint action of $ G$ on $\fg$. Then applying \eqref{p} and by direct calculation,
\[(A\cdot f)(\Psi)=\frac{d}{dt}(f(\Psi+t[A,\Psi]))|_{t=0}=\frac{d}{dt}(f(\Psi+t\partial_z\Psi))|_{t=0}\]
is equal to $\partial_zf(\Psi)$.
Now if $f$ is an invariant polynomial, we have
$A\cdot f=0$. Therefore $(\partial_zf)(\Psi)=0$.}
\end{proof}

\begin{lemma}\label{Horizontal} 
For a line bundle with a connection $(\mL,\nabla_\mL)$ on $C$ and a horizontal section
$\phi\in\Gamma(C,\ad (E)\otimes\mL)$ (with respect to the tensor connection), 
the section $h_{\mL}(E,\phi)\in\Gamma(C,\fc_{\mL})$ is a horizontal section of 
$\fc_{\mL}$ (with respect to the connection induced from $\nabla_\mL$).
\end{lemma}
\begin{proof}
This follows from the fact that $\ad (E)\otimes\mL\ra\fc_\mL$ is a horizontal morphism of $\mD_C$-schemes.
\end{proof}
We call $h_p$ the $p$-Hitchin fibration. Its fiber over $b'$ is denoted by $\Loc_{G,b'}$. 

\medskip

Similarly to the usual Hitchin fibration, the $p$-Hitchin fibration
has a large symmetry. We need a lemma. Let $(E,\nabla)$ be a $G$-local system on $C$ and $\Psi=\Psi(\nabla)\in \ad(E)\otimes \omega^p$ be the $p$-curvature. It induces a morphism $a_{E,\Psi}:J_{b^p}\to \Aut(E)$ as in \eqref{actEphi}, where $b^p=h_{\omega^p}(E,\Psi)$ is equal to $F^p(b')$ for some $b'\in B'$. Observe that $J_{b^p}=F^*J'_{b'}$ is a $\mD_C$-group
scheme so that $( J_{b^p})^{\nabla}=J'_{b'}$.
We refer to \S \ref{A} for the detailed discussion of $\mD_C$-group schemes.
\begin{lem}\label{abel:hor}
The morphism $a_{E,\Psi}$ is horizontal.
\end{lem}
\begin{proof}We consider a more general situation. 
Let $E\times \mL^\times$ be a $G\times\bbG_m$-local system on $C$. Observe that if $f:X\to Y$ is any $G\times \bbG_m$-equivariant morphism of schemes, then the twist of $f$ by $E\times\mL^\times$, denoted by $f_{E\times\mL^\times}: X_{E\times\mL^\times}\to Y_{E\times\mL^\times}$ is a horizontal morphism of $\mD_C$-schemes. If we apply the above construction to the following diagram
\[\xymatrix{G\times\fg\ar[dr]& \chi^*J\ar_{a}[l]\ar[r]\ar[d]& J\ar[d]\\
&\fg\ar^\chi[r]& \fc.
}
\]
where $a:\chi^*J\to G\times \fg$ is as in Proposition \ref{J} (1), we obtain a horizontal morphism $a_{E\times\mL^\times}:\chi_{E\times\mL^\times}^*J_\mL\to \Aut(E)$ of group schemes over the total space of $\ad(E)\otimes\mL$. If
$\phi: C\to \ad(E)\otimes\mL$ is a horizontal section, then $b_{\mL}= h_\mL(E,\phi)\in B_\mL$ is a horizontal section of $\fc_\mL$ by Lemma \ref{Horizontal}, and the pullback of $a_{E\times\mL^\times}$ along $\phi$ is $a_{E,\phi}: J_{b_\mL}\to \Aut(E,\phi)$ as in \eqref{actEphi}, which is horizontal.

Finally, the lemma follows by applying the construction to $\mL=\omega^p$ with the canonical connection, and $\phi=\Psi=\Psi(\nabla)$. 
\end{proof}

Note that the $\mD_C$-group schemes $J_{b^p}$ can be organized into 
a family over $C\times B'$, denoted by $J^p\ra C\times B'$. Namely, it is the pull back of  
the regular centralizers $J'_{\omega'}\ra\fc'_{\omega'}$
along the map $C\times B'\stackrel{F\times \id}\ra C'\times B'\stackrel{}\ra\fc'_{\omega'}$. 
Then
$J^p$ is naturally a $\mD_{C\times B'/B'}$-group scheme (cf. Example \ref{Cartier}) and 
for any $b'\in B'$, the restriction $J^p|_{C\times\{b'\}}$ is isomorphic to $J_{b^p}$.
Therefore given
$(P,\nabla)\in \Loc_{J_{b^p}}$, one can apply the formulation \S \ref{some functors} to form a
$G$-local system
\[((P,\nabla),(E,\nabla))\mapsto P\otimes E:=(a_{E,\Psi})_*P\otimes E.\]
It is easy to see that this induces a morphism
$\otimes:\Loc_{J^p}\times_{B'}\Loc_{G}\to \Loc_G$. 
\begin{lem}Assume that $h_p(E,\nabla)=b'$.
Let $P'$ be a $J'_{b'}$-torsor, and $P=F^*P'$ be the $J_{b^p}$-bundle with the canonical connection. Then $h_p(P\otimes E,\nabla_{P\otimes E})=b'$.
\end{lem}
The question is local so one can assume that bundles are trivial. Then it is a direct calculation.
We refer to Lemma \ref{J to G} for more details, where a similar statement is proved. We identify $\sP'$ with the substack of $\Loc_{J^p}$ with vanishing $p$-curvature. Then we obtain
\begin{prop}
There is a natural action morphism $\sP'\times_{B'}\Loc_G\to \Loc_G$ of stacks over $B'$.
\end{prop}

\begin{remark}
Alternatively, one can see the action of $\sP'$ on $\Loc_G$ as follows. One can define a sheaf $\underline\Aut(E,\nabla)$ on $X'_{et}$
by assigning every \'{e}tale map $f:U'\to X'$ the group
\[\underline{\Aut}(E,\nabla)(U')=\Aut(f^*E,f^*\nabla).\]
One can show that there is an isomorphism $\Aut(E)^\nabla=\underline{\Aut}(E,\nabla)$ as \'{e}tale sheaves on $X'$. Then it follows from Lemma \ref{abel:hor} that there is an action of $J'_{b'}$ on $(E,\nabla)$, and therefore we can twist $(E,\nabla)$ by a $J'_{b'}$-torsor. 
\end{remark}

\subsection{The stack $\sH$}\label{locJtau}
We will show that the $p$-Hitchin map
$h_p:\Loc\to B'$ is closely related to the classical Hitchin map
$h':\Hg'\to B'$. Namely, they are related by a $\sP'$-torsor
$\sH$.

Let $b'\in B'$, and $J'_{b'}$ be the group scheme over $C'$. Then $F^*_CJ'_{b'}=J_{b^p}$ is a $\mD_C$-group scheme. Therefore, it makes sense to define the stack $\Loc_{J_{b^p}}$ of $J_{b^p}$-local
systems on $C$, where $b'\in B'$. This is representable 
because the map $\Loc_{J_{b^p}}$ to the Picard stack $\sP_{b^p}$  of $J_{b^p}$-torsors on $C$ is schematic. By Proposition \ref{p-hitchin} (or \S
\ref{The commutative case}), there is the $p$-Hitchin map
\begin{equation}\label{locJtau:p}
h_{J_{b'}}:\Loc_{J_{b^p}}\to B'_{J'_{b'}}:=\Gamma(C',\Lie
J_{b'}\otimes \omega_{C'}).
\end{equation}
Recall that $B'_{J'_{b'}}$ is just the fiber of the vector bundle
$B'_{J',\omega'}$ over $B'$ introduced in \S \ref{tau and c}. It is
convenient to work with the universal situation, i.e. $b':B'\to B'$
is the identity map. Then the $p$-Hitchin map can be organized into a morphism
\[h_{J'}:\Loc_{J^p}\to B'_{J',\omega'},\]
over $B'$.
The result in \ref{tangent map} implies
\begin{lemma}\label{smoothness}
This morphism is smooth.
\end{lemma}
\begin{proof}
First, both $\Loc_{J^p}$ and $B'_{J',\omega'}$ are smooth Picard stacks over $B'$. Namely, the latter is a vector bundle over $B'$ and the former is an extension of the smooth Picard stack $\sP_{\omega^p}\times_{B_{\omega^p}}B'$ by a vector group. Therefore we can check the smoothness of $h_{J'}$ fiberwise over $B'$. In addition, the map $h_{J'}$ is additive by \eqref{appen:add}. 
Now Lemma \ref{appen:tangent} identifies the tangent map $dh_{J'}$ at the unit with the edge map $c: E^1_\infty\to E^{0,1}_2$ in the ``second spectral sequence" that calculating the hypercohomology of the de Rham complex $(\Lie J_{b^p}\to \Lie J_{b^p}\otimes \omega_C)$. But the latter is surjective because $E_2^{2,0}\is H^2(C',\Lie J'_{b'})=0$.
\end{proof}

Recall the section $\tau': B'\to B'_{J',\omega'}$ defined in \eqref{taut
sect II}. Let us define 
\[\sH=\Loc_{J^p}(\tau'):=B'\times_{B_{J',\omega'}'}\Loc_{J^p},\]
the base change of $\Loc_{J^p}\ra B'_{J',\omega'}$ along $\tau'$. 
By Corollary \ref{commutative Cartier}, this is a pseudo $\sP'$-torsor. An object in $\sH$ is a $J^p$-local system $(P,\nabla)$ with a specific $p$-curvature. We have the following theorem whose proof will be given in \S\ref{proof}.

\begin{thm}\label{sH torsor}
The stack $\sH$ is a $\sP'$-torsor.
\end{thm}


\subsection{The main theorem}\label{NC}
We will prove our main theorem in this subsection, assuming Theorem \ref{sH torsor} whose proof will be given in the next subsection.

We first construct
\begin{prop-dfn}\label{twist}
There is a canonical morphism of stacks over $B'$
\[\fC:\sH\times^{\sP'}\Hg'_G\to \Loc_G.\]
\end{prop-dfn}
\begin{proof}
The construction of the morphism $\fC$ is given as follows. For any $(E',\phi')\in
\Hg'_{b'}$, there is a canonical morphism $a_{E',\phi'}:J'_{b'}\to
\Aut (E',\phi')\subset \Aut(E')$ (see\eqref{actEphi}) and therefore
via pullback there is a horizontal morphism of $\mD_C$-groups
$$F^*a_{E',\phi'}=a_{F^*E',F^*\phi'}: J_{b^p}\to \Aut(F^*E'),$$
where $F^*E'$ is equipped with the canonical connection. 
Since $F^*E'$ is naturally a $\Aut(F^*E')\times G$-local system (see Example \ref{autE}),
by the formalism of induction and tensor functors in \S
\ref{some functors},
given
$(P,\nabla)\in \Loc_{J_{b^p}}$, one can form a
$G$-local system
\[((P,\nabla),(E',\phi'))\mapsto (F^*a_{E',\phi'})_*P\otimes F^*E'.\]
It is easy to see that this induces a morphism
$\fC: \Loc_{J^p}\times^{\sP'}\Hg'\to \Loc_G$. Indeed, let $P'$ be a $J'_{b'}$-torsor on $C'$. By \eqref{actEphi}, one can twist the Higgs field $(E',\phi')$ by $P'$ and let $P'\otimes (E',\phi')$ denote the new Higgs field. Then the claim amounts to
\[(F^*a_{E',\phi'})_*(P\otimes F^*P')\otimes F^*E'= (F^*a_{P'\otimes(E',\phi')})_*P\otimes F^*E',\]
which can be checked directly by definitions.

To show that $\fC$ is a morphism over $B'$, it is enough to show the following lemma. \end{proof}

\begin{lem}\label{J to G}
If $(E',\phi')\in \Hg'_{b'}$ and $(P,\nabla)\in \Loc_{J_{b^p}}$
whose $p$-curvature is $\tau'(b')\in B'_{J'_{b'}}$, then
$h_p((F^*a_{E',\phi'})_*P\otimes F^*E')=b'$.
\end{lem}
\begin{proof}Let us regard $b':C'\to \fc'_{\omega}$ as a section. Then
the question is local on $C$. So we can assume that $E'$ and $P$ are
trivialized, and denote the map $F^*a_{E',\phi'}:J_{b^p}\to
\Aut(F^*E')$ by $F^*a_{\phi'}:J_{b^p}\to G\times C$. This is exactly
the map $J_{b^p}\simeq I_{F^*\phi'}\to G\times C$.

By definition, the $p$-curvature of the $G$-local system
$F^*E'=G\times C$ is zero, and the $p$-curvature of $(P,\nabla_P)$
is $\Psi(\nabla_P)=F^*\tau(b')\in\Lie J_{b^p}\otimes
F^*\omega_{C'}$. By the diagram \eqref{p-curv for}, it is easy
to see that the $p$-curvature of $(F^*a_{E',\phi'})_*P\otimes F^*E'$ is
given by the image of $\tau(b^p)=F^*\tau'(b')$ under
$dF^*a_{\phi'}:\Lie J_{b^p}\otimes F^*\omega_{C'}\to \fg\otimes
F^*\omega_{C'}$. But by \S \ref{tau and c}, 
$dF^*a_{\phi'}(\tau(b^p))=F^*\phi'$, and therefore its image under $\fg_{\omega^p}\to \fc_{\omega^p}$ is
$b^p$. The lemma follows.
\end{proof}
\begin{remark}It is clear from the construction that $\fC$ is $\sP'$-equivariant.
\end{remark}

Our main theorem is
\begin{thm}\label{main}
The morphism $\fC$ is an isomorphism.
\end{thm}
\begin{proof}
To prove that $\fC$ is an isomorphism, we will construct the inverse morphism. This is essentially explained in the introduction: Given a $G$-local system $(E, \nabla)$, a solution of the equation \eqref{globhe} defines a Higgs bundle $(E',\phi')$. Here we make it precise.
Namely, by Corollary \ref{commutative Cartier}, the pseudo $\sP'$-torsor $\Loc_{J^p}(-\tau')$ of $J^p$-local systems with $p$-curvature $-\tau'$ is the inverse of $\sH=\Loc_{J^p}(\tau')$. In particular, it is a $\sP'$-torsor. Therefore, it is enough to construct a morphism
\[\fC^{-1}: {\sH}^{-1}\times^{\sP'}\Loc_G\to \Hg'_G.\]

Let $(P,\nabla_P)$ be a $J_{b^p}$-local system with the $p$-curvature $-\tau(b^p)$ so that $\tilde{E}=(a_{E,\Psi})_*P\otimes E$ is a $G$-local system. In addition, there is an F-Higgs field $\tilde{\Psi}$ on $\tilde{E}$. Indeed, $(\tilde{E},\tilde\Psi)$ is the twist of $(E,\Psi)$ by the underlying $J_{b^p}$-torsor $P$. In particular, under the classical Hitchin map $h_{\omega^p}:\Hg_{G,\omega^p}\to B_{\omega^p}$, $h_{\omega^p}(\tilde E,\tilde\Psi)=b^p$.

We show that the connection on the $G$-local system $\tilde{E}$ has vanishing $p$-curvature and $\tilde\Psi$ is horizontal. Then $(\tilde{E},\tilde\Psi)=F^*(E',\psi')$, and $h'(E',\psi')=b'$. This construction provides the inverse map of $\fC$.

The question is local on $C$. We can trivialize $P$, $E$ and $\omega_C$. Then $\Psi\in \fg\otimes \mO_C, \Psi_P\in\Lie J_{b^p}$, and by Lemma \ref{taut:prolong}, $da_{E,\Psi}: \Lie J_{b^p}\to \fg\otimes\mO_C$ will send $\Psi_P$ to $-\Psi$.

Note that the $p$-curvature of $\tilde{E}$ with respect to its connection $\tilde\nabla$ (do not confuse with $\tilde\Psi$) is given by
$\Psi_P\otimes 1+1\otimes \Psi\in \Der_{\mO_C}((\mO_P\otimes\mO_E)^{J_{b^p}})$. Here we use the fact that $\Psi_P\otimes 1+1\otimes \Psi$, which is a priori an element in $\Der_{\mO_C}(\mO_P\otimes\mO_E)$, preserves $(\mO_P\otimes\mO_E)^{J_{b^p}}$. 
It is clear that under the isomorphism $(\mO_P\otimes\mO_E)^{J_{b^p}}\simeq \mO_E$ induced by the trivialization, $\Psi_P\otimes 1$ maps to $da_{E,\Psi}(\Psi_P)$. Therefore, $\Psi_P\otimes 1+1\otimes \Psi=0$.

Finally, we show that $\tilde{\Psi}$ is horizontal. Again, the question is local and we pick up the trivialization of $P$. Then under the isomorphism $\tilde{E}\simeq E$ of $G$-bundles, $\tilde{\nabla}_z=\nabla_z+A$ for some $A\in\on{Im}(a_{E,\Psi}:\Lie J_{b^p}\to \fg\otimes\mO_C)$. Then the claim follows from $[A,\Psi]=0$ and $\nabla_z(\Psi)=0$.
\end{proof}

\begin{remark}\label{ex:GLn}
It is instructive to look at the isomorphism $\fC$ in the case when $G=\GL_n$. For simplicity, we fix 
$b'\in B'$. We first describe the fibers $\Hg'_{G,b'}, \Loc_{G,b'}$ and $\sH_{b'}$ in this case.

Let $S'_{b'}\in
T^*C'$ be the spectral curve for $b'$, and $S_{b^p}\subset
T^*C'\times_{C'} C$ be the pullback of $S'_{b'}$ which fits into the following Cartesian diagram
\[\begin{CD}
S_{b^p}@>W>> S'_{b'}\\
@V\pi_{b^p}VV@VVV\\
C@>>>C'.
\end{CD}\]
Note that $S_{b^p}$ is the
spectral curve corresponding to $b^p=F^*b'$. 

Then $\Hg'_{G,b'}$ is the space of Higgs fields $(E',\phi')$ with the characteristic polynomial $b'$. By the classical BNR correspondence (cf. \cite{BNR}), such a Higgs field defines  a coherent sheaf $\mF_{(E',\phi')}$ on $S'_{b'}$. 
The space $\Loc_{G,b'}$ consist of  rank $n$
vector bundles  with a connection $(E, \nabla)$ on $C$ whose $p$-curvature $\Psi(\nabla)$ has the characteristic polynomial $b^p$. Therefore, every object in $\Loc_{G,b'}$ defines a coherent sheaf $\mF_{(E,\Psi(\nabla))}$ on $S_{b^p}$.
Finally,
$$J_{b^p}=\pi_{b^p,*}\bbG_m,$$ 
and
$\sH_{b'}$ can be regarded as the open substack of $\Loc_{G,b'}$ consisting of those $(E, \nabla)$, such that $\mL:=\mF_{(E,\Psi(\nabla))}$ is an invertible sheaf on $S_{b^p}$. Note that for such a pair $(E,\nabla)\in \sH_{b'}$, the
direct image of $\mL$ along $W:S_{b^p}\to S'_{b'}$ is locally free of
rank $p$ on $S'_{b'}$ and therefore is a splitting module of the
restriction  to ${S'_{b'}}$ of the
Azumaya algebra of the ring of crystalline differential operators $\mD_C$. As a result, the map 
$$\fC_{b'}: \sH_{b'}\times^{\sP'_{b'}}\Hg'_{G,b'}\to \Loc_{G,b'}$$
will send an object $(E',\phi')\in \Hg'_{G,b'}$, regarded as a
coherent sheaf $\mF_{(E',\phi')}$ on $S'_{b'}$, and an object $\sH_{b'}$, which defines a splitting module $W_*\mL$ of
$\mD_{C}|_{S'_{b'}}$, to the tensor product $\mF_{(E',\phi')}\otimes W_*\mL$.  
\end{remark}

\begin{remark}The description of $\sH_{b'}$ as a natural open substack of $\Loc_{G,b'}$ in the above example is a special feature for $G=\GL_n$. It is due to the existence of a natural open embedding $\sP'_{b'}=\Pic(S'_{b'})\subset \Hg'_{G,b'}$. 

In general, let $\Hg^{reg}_{G,\mL}$ denote the open substack of $\Hg_{G,\mL}$ consisting of
$(E,\phi):C\to [\fg_\mL/G]$ that factor through $C\to
[(\fg^{reg})_\mL/G]$. It is known from \cite{DG,N1} that
 $\on{Higgs}^{reg}_{\mL}$ is a $\sP_{\mL}$-torsor.  
Then a choice of the trivialization of this torsor (e.g. by a choice of Kostant section $\epsilon_{\mL^{1/2}}$) defines an open embedding of $\sP_{\mL}\to \Hg^{reg}_{G,\mL}$.

Let us define an open substack $\Loc_G^{reg}$ of $\Loc_G$ consisting of those $(E,\nabla)$ such that the $F$-Higgs field $(E,\Psi(\nabla))\in \Hg_{G,\omega^p}^{reg}$. Then $\fC$ restricts to an isomorphism 
$$\sH\times^{\sP'}{\Hg'}^{reg}_{G}\simeq \Loc_G^{reg}.$$
Therefore, a trivialization of ${\Hg'}^{reg}_G$ as the $\sP'$-torsor also defines an embedding $\sH\simeq \Loc_{G}^{reg}\subset \Loc_G$. Note that as a corollary, we see that 
the map $\Loc^{reg}_G\to B'$ is surjective, which is not obvious from its definition.
\end{remark}

\subsection{Proof of Theorem \ref{sH torsor}}\label{proof}
This subsection is devoted to the proof of theorem \ref{sH torsor}.

By Lemma \ref{smoothness}, $\sH$ is smooth over $B'$. So it is enough to show that $\sH\ra B'$ is surjective. We will fix $b'\in B'(k)$.
Let $\sG_{\tau'}$ be the $J'$-gerbe corresponding to the section $\tau'(b')\in\Gamma(C',\Lie J'\otimes\omega_{C'})$ via the four term short exact sequence in Proposition \ref{appen:4-term}.
Then $\sH_{b'}$ is just the pseudo $\Bun_{J'}$-torsor 
of splittings of $\sG_{\tau'}$. 
We denote by $[b']\in H^2(C',J_{b'}')$ the class corresponding to the gerbe $\sG_{\tau'_{b'}}$. We need to show that $[b']$ is trivial.
We begin with the following lemma.
\begin{lemma}
The class $[b']$ lies in the image of the map $H^2(C',(J'_{b'})^0)\ra H^2(C',J_{b'}')$. 
Here $(J_{b'}')^0$ is the neutral component of $J_{b'}'$.
\end{lemma}
\begin{proof}
By definition, the class $[b']$ is obtained from $\tau'(b')$ via the exact sequence in Proposition \ref{appen:4-term} applied to $J'_{b'}$. But we can also apply this sequence to $(J_{b'})^0$ to produce a class in $H^2(C',(J'_{b'})^0)$. Clearly, this new class will map to $[b']$. 
\end{proof}

\begin{prop}\label{vanishing}
For any $b\in B(k)$ we have $H^2(C,J_b^0)=0$.
\end{prop}
Clearly this proposition implies that $[b']$ is zero, and thus finish the proof of
surjectivity of $\sH\ra B'$. So it is enough to prove this proposition.

\begin{proof}
Let $K$ be the function field of $C$ and let
$j:\eta=\Spec K\ra C$ be the inclusion.
Fix $b\in B(k)$ and we write  $J_\eta^0:=J^0_b|_\eta$.
The group $J_\eta^0$ is smooth connected and commutative. Moreover,
for a choice of a trivialization of $\omega_C^\times$ at $\eta$ we have
$J_\eta^0\is (G_\eta^x)^0$, where $x\in\fg^{reg}(K)$ is the image of $b\in \fc(K)\stackrel{kos}{\to}\fg(K)$, and 
$(G_\eta^x)^0$ is the neutral component of the centralizer of $x$ in $G_\eta=G\times_kK$.
There exists the largest affine subgroup $J_{\eta,s}^0\subset J_\eta^0$ of multiplicative type such 
that the quotient $U:=J_\eta^0/J_{\eta,s}^0$ is unipotent.
Recall that a unipotent group over a field $K$ is called $K$-split if it admits 
a composition series with successive quotients $K$-isomorphic to $\bG_a$.

\begin{lemma}
The unipotent group $U$ is $K$-split.
\end{lemma}
\begin{proof}
Let $F$ be the separable closure of $K$. 
By \cite[Theorem B.3.4]{CGP},
it is enough to show that $U_F:=U\otimes_K F$
is $F$-split. 
We first construct a $\bG_m$-action on $U_F$ defined over $F$. 
Let $x=x_s+x_u$ be the Jordan decomposition
and $L_{\bar\eta}:=G_{\bar\eta}^{x_s}$ be the centralizer of $x_s$ in 
$G_{\bar\eta}=G\times_k\bar F$
\footnote{Notice that $x_s,x_u\in\fg(\bar F)$ are not necessary in $\fg(F)$ since $F$ is not perfect.}.
\quash{
The Levi subgroup $L_{\bar F}$ is defined over $F$, i.e. there is a levi subgroup $L\subset G$ such that $L_{\bar F}=L\otimes_F\bar F$}It is known that $x_u\in\Lie L_{\bar\eta}$ is regular nilpotent and we have 
\[J_{\bar\eta}^0:=J_{F}^0\otimes_F\bar F\is (G_{\bar\eta}^{x})^0=(L_{\bar\eta}^{x_u})^0.\]
There exists a co-character $2\rho_L:\bG_m\ra L_{\bar\eta}$
such that $\on{Ad}(2\rho_L(t))(x_u)=t^2x_u$. The co-character $2\rho_L$
defines a $\bG_m$-action $2\rho_L(\bG_m):\bG_m\times
J_{\bar\eta}^0\ra J_{\bar\eta}^0$ on $J_{\bar\eta}^0$ by conjugation action.
We claim that the action map $2\rho_L(\bG_m)$ is defined over $F$. 
Since the closed embedding $i:J_{\bar\eta}^0\hookrightarrow G_{\bar\eta}$
is defined over $F$,
it is enough to show that the composition 
\[a:\bG_m\times
J_{\bar\eta}^0\ra J_{\bar\eta}^0\hookrightarrow G_{\bar\eta}\]
is defined over $F$. But in fact, $a$ factors as
\[\bG_m\times J_{\bar\eta}^0\stackrel{2\rho_{L,G}\times i}{\longrightarrow} G_{\bar\eta}\times G_{\bar\eta}\stackrel{\on{Ad}}{\to}G_{\bar\eta}
\]
where $2\rho_{L,G}$ is the composition $2\rho_{L,G}:\bG_m\stackrel{2\rho_L}\ra L_{\bar\eta}\hookrightarrow G_{\bar\eta}$, which is defined over $F$ since it factors through a maximal 
torus of $G_{\bar\eta}$ and any maximal torus of $G_{\bar\eta}$ and  any
homomorphism between tori are defined over $F$.
Since the adjoint action $\on{Ad}$ of $G_{\bar\eta}$ is also defined over $F$, so is $a$. This finishes the construction of the promised $\bG_m$-action on $J_F^0$ over $F$.
Clearly, this 
$\bG_m$-action preserves $J_{F,s}^0$ and therefore induces a $\bG_m$-action
on $U_F\is J^0_{F}/J_{F,s}^0$.
By abuse of notation we still denote the action by $2\rho_L(\bG_m)$.

We next show that the action $2\rho_L(\bG_m)$ on $U_F$ 
has only nontrivial weights on $\Lie\!\ U_F$. It is enough to prove this statement over $\bar F$.
Let $B_L$ be the Borel subgroup of $L_{\bar\eta}$ such that $x_u\in\Lie\!\ U_L$, where $U_L$ is the unipotent radical of $B_L$.
Without loss of generality, we can assume $x_u=\sum_{\alpha\in\Delta_L} x_{\alpha}$ 
and $2\rho_L=\sum_{\alpha\in\Delta_L}\alpha^\vee$,
where $\Delta_L$ is the set of simple roots of $L_{\bar\eta}$ determined by $B_L$
and $x_\alpha$ is a non-zero element in the corresponding root space. We have $U_{\bar F}=U_L^{x_u}\subset U_L$ and the statement follows 
from the fact that the action $2\rho_L(\bG_m)$ on $U_{\bar F}$ is the restriction of 
the conjugate action of $2\rho_L$ on $U_L$ and the later action has 
only nontrivial weights on $\Lie\!\ U_L$.

We have constructed an action of $\bG_m$ on $U_F$ over $F$
with only nontrivial weights on $\Lie\!\ U_F$. Then the lemma will follow from the general result.
\end{proof}

\begin{lem}
Let $U$ be a smooth connected commutative affine unipotent group over a 
separably closed 
field $F$ of $char F=p>0$. If we have an action $\chi(\bG_m)$ of $\bG_m$ on $U$ over $F$ with only nontrivial weights on $\Lie U$, then $U$ is $F$-split. 
\end{lem}
\begin{proof}

\quash{
Since $U$ is unipotent there exists an integer $n>0$ such that $p^nU=0$.
Consider the filtration $U_i:=ker(p^i:U\stackrel{}\ra U)$ of $U$. 
Each $U_{i+1}/U_i$ is $p$-torsion commutative affine $F$-group. Moreover,
the $\chi(\bG_m)$-action
preserves each $U_i$ and induces a $\bG_m$-action on $U_{i+1}/U_i$ (which we still denote by
$\chi(\bG_m)$).
We claim that for each $i\geq 0$ we have $(U_{i+1}/U_i)^{\chi}=e$ .
To see this, let $\bar u\in(U_{i+1}/U_i)^{\chi}$. Then for any lifting $u\in U_{i+1}$ of $\bar u$
we have $p^iu\in U^{\chi}$. Indeed, we have $\chi(t)u=u+y(t)$ where  
$y(t)\in U_i$, thus $\chi(t)(p^iu)=p^i(\chi(t)u)=p^i(u+y(t))=p^iu$. By assumption 
$U^{\chi}=e$, hence $u\in U_i$ and $\bar u=e$. 
any  $p$-torsion commutative affine $F$-group
which equips a $\bG_m$-action over $F$ with no nontrivial invariant subgroups is isomorphic to a vector group, therefore
we have $U_{i+1}/U_i\is\bG_a^l$ over $F$ and it implies $U$ is $F$-split.
}
Now by \cite[Theorem B.3.4]{CGP}, 
there is a unique smooth connected $F$-split $F$-subgroup $U_{s}$
of $U$ such that $U_{w}:=U/U_{s}$ is $F$-wound (see \emph{loc. cit.} for the definition of 
$F$-wound). 
We need to show that $U_w$ is trivial.
By the uniqueness of $U_{s}$, the action $\chi(\bG_m)$
preserves $U_{s}$ and induces a $\bG_m$-action on $U_{w}$.
On the other hand, by \cite[Theorem B.4.3]{CGP}, any $\bG_m$-action 
over $F$ on an $F$-wound
smooth connected unipotent group is trivial. Thus, $\bG_m$ acts trivially on $U_{w}$ and its 
Lie algebra $\Lie\!\ U_w$ belongs to the zero weight space of $\Lie\!\ U$. Therefore, by assumption. $\Lie\!\ U_w=0$. Then $U_w$ is trivial since it
is smooth and connected.
\end{proof}

\begin{lem}\label{vanishing 2}
We have $H^i(\Spec K,J_\eta^0)=0$ for $i\geq 1$.
\end{lem}
\begin{proof}
We have an exact sequence 
$H^i(\Spec K,J_{\eta,s}^0)\ra H^i(\Spec K,J_\eta^0)\ra H^i(\Spec K,U)$.
Thus it suffices to show that $H^i(\Spec K,J_{\eta,s}^0)=H^i(\Spec K,U)=0$
for $i\geq 1$.
Since $\on{dim}K\leq 1$ and $J_{\eta,s}^0$ is a 
connected torus, we have $H^i(\Spec K,J_{\eta,s}^0)=0$
for $i\geq 1$. On the other hand, $U$ is $K$-split we have 
$H^i(\Spec K,U)=0$ for $i\geq 1$. We are done.
\end{proof}

Finally, we prove Proposition \ref{vanishing}.
Since $C$ is a curve we have $H^2(C,J_b^0)\is H^2(\Spec K,j_*J_\eta^0)$.
In addition, $Rj_*^iJ_\eta^0=0$ for $i\geq 1$. Indeed, the $i$-th direct image
is the sheafification of the functor $S\ra H^i(\Spec K\times_CS,J_\eta^0).$
Now $\Spec K\times_CS$ is the spectrum of a finite \'{e}tale $K$-algebra, which is a finite product of finite separable extensions of $K$.
Then the cohomology vanishes by Lemma \ref{vanishing 2}. 
By Leray spectral sequence, this gives us \[H^2(C,J_b^0)\is H^2(\Spec K,j_*J_\eta^0)\is 
H^2(\Spec K,J_\eta^0),\] 
which vanishes by Lemma \ref{vanishing 2}.
This finishes the proof.
\end{proof}
\begin{example}
Let us look at the most singular case $b=0\in B(k)$. We assume that $G$ is semisimple for simplicity.
Then $J_b=G^e\times^{\bG_m}\omega_C^\times$, where $e\in\fg^{reg}$ is regular nilpotent. 
So the group $J_\eta^0$ is isomorphic to
$J_\eta^0\is (G^e)^0\otimes_kK$.
As the group $(G^e)^0$ is 
smooth connected and unipotent over an algebraically closed field $k$, $J_\eta^0$ is $K$-split.
\end{example}

\subsection{Trivialization of $\sH_0$}\label{triv}
We briefly study trivializations of the $\sP'$-torsor $\sH$ over $0\in B'$. We first deal the case $G=\on{PGL}_2$, which is closely related to the geometry of the curve itself.
\begin{lem}
If $G=\on{PGL}_2$, then $J'_0\simeq \omega^{-1}_{C'}$, where $\omega^{-1}_{C'}$ is regarded as an affine vector group over $C'$.
\end{lem}
As a result, $\sP'_0\simeq H^1(C',\omega_{C'}^{-1})$, and $\sH_0$ is an $H^1(C',\omega_{C'}^{-1})$-torsor. In addition, for $b'=0$, $J_{b^p}=F^*\omega_{C'}^{-1}=\omega_C^{-p}$. So we have
\begin{lem}\label{ex H0}
The stack $\sH_0$ consists of  $F^*\omega_{C'}^{-1}$-torsors $E$ equipped with a connection $\nabla$, such that $h_p(E,\nabla)=1$ as an element in $\Gamma(C', \Lie\!\ J'_0\otimes\omega_{C'})=\Gamma(C',\omega_{C'}^{-1}\otimes\omega_{C'})=\Gamma(C',\mO_{C'})$.
\end{lem}

\begin{lem}\label{PGL_2}
Under the above identification, the $\sP'_0$-torsor $\sH_0$ is canonically isomorphic to the $H^1(C',\omega_{C'}^{-1})$-torsor of liftings of $C'$ to $W_2(k)$.
\end{lem}
\begin{proof}This is essentially a reformulation of \cite[Theorem 4.5]{OV} in the curve case. Given a lifting $\tC'$ of $C'$ to $W_2(k)$, we construct an object in $\sH_0$ as follows.
As explained in \cite[\S 1]{OV}, a lifting of $C'$ to $W_2(k)$ defines an extension of  $\mO_{C}$ by $F^*\omega^{-1}_{C'}$ as $\mD_C$-modules
\[0\to F^*\omega^{-1}_{C'}\to \mE\stackrel{\pi}{\to} \mO_C\to 0 .\]
such that the $p$-curvature of $\mE$ is given by
\[\mE\to \mO_C\simeq F^*\omega_{C'}^{-1}\otimes F^*\omega_{C'}\subset \mE\otimes F^*\omega_{C'}.\]
Then $\pi^{-1}(1)$ is an $F^*\omega_{C'}^{-1}$-torsor on $C$, equipped with a connection\footnote{This is in fact the torsor of liftings of the Frobenius $F:C\to C'$ to $W_2(k)$.}. By Lemma \ref{ex H0}, this defines an object of $\sH_0$. 

This construction induces a morphism from the $H^1(C',\omega_{C'}^{-1})$-torsor of liftings of $C'$ to $W_2(k)$ to the $\sP'_0$-torsor $\sH_0$, which clearly intertwines the action of $H^1(C',\omega_{C'}^{-1})\simeq \sP'_0$.  
\end{proof}
As a corollary, in the case $G=\on{PGL}_2$ a choice of the lifting of $C'$ to $W_2(k)$ gives rise to a trivialization of $\sH_0$. The same is true for $G=\SL_2$, as $J'_0\simeq \omega^{-1}_{C'}\times\mu_2$. 
Now for a general reductive group $G$, we fix a principal $\varphi: \SL_2\to G$. Via pushout, we see
that a lifting of $C'$ to $W_2(k)$ gives rise to a trivialization of $\sH_0$ as well.
\begin{remark}\label{other lifting}
If the group 
$G$ is bigger than $\on{SL}_2$, then there are trivializations of $\sH_0$
that do not come from liftings of $C'$ to $W_2(k)$, as the dimension of 
$\sH_0$ is bigger than the dimension of the  $H^1(C',\omega_{C'}^{-1})$-torsor of liftings of $C'$ to $W_2(k)$. Therefore, for group $G$ bigger than $\on{SL}_2$, there are some Simpson correspondences that do not arise from the construction in \cite{OV}.
\end{remark}


\appendix

\section{The Stack of $\mG$-local systems}\label{A}
In this appendix we discuss the notion of (de Rham) $\mG$-local systems and
their $p$-curvatures. We will fix a smooth morphism $X\to S$ of
noetherian schemes (however all discussions carry through without change if $X$ is a smooth Deligne-Mumford stack), and a smooth affine group scheme $\mG$ over $X$. We do not assume that $\mG$ is constant or is fiberwise connected. The main example is
the regular centralizer group scheme $\mJ_{b^p}$ as in the note. 

For any scheme $X\to S$ smooth over $S$, we
denote by $T_{X/S}$ (resp. $\Omega_{X/S}$) its tangent (resp.
cotangent) sheaf  relative to $S$, or sometimes by $T_X$ (resp.
$\Omega_X$) if no confusion will likely arise. We denote by $\mD_{X/S}$ (or by $\mD_X$ for simplicity) the sheaf of crystalline differential operators on $X$ as in \cite{BB,OV}.


\subsection{Connections on $\mG$-torsors}\label{quasi-coh}
In order to talk about a connection on a $\mG$-torsor, we need to
assume that $\mG$ itself is a $\mD_{X/S}$-group scheme. I.e. it is equipped with a flat connection
\[\nabla_{\mG}:\mO_{\mG}\ra\mO_{\mG}\otimes_{\mO_X}\Omega_{X/S}\]
which is compatible with the unit, the multiplication and the
co-multiplication on $\mO_{\mG}$. Equivalently, let
$\Delta:X\ra X\times _SX$ be the diagonal embedding and let
$\Delta^{(1)}$ be the first  infinitesimal neighborhood of $\Delta$ with two natural projections $p_1,p_2: \Delta^{(1)}\to X$. Then a
flat connection on $\mG$ is an isomorphism of group schemes
$\nabla_{\mG}:p_1^*(\mG)\is p_2^*(\mG)$ that restricts to the identity
map on $\Delta$ and that satisfies the usual cocycle condition when pull back to the first infinitesimal neighborhood $\Delta_3^{(1)}$ of the main diagonal $\Delta_3: X\to X\times X\times X$.

Given a $\mG$-torsor $E$, a flat connection on $E$ is a $\mD_{X/S}$-scheme structure on $E$ that is compatible with the $\mG$-action. Explicitly, it is a flat connection $\nabla:\mO_E\ra\mO_E\otimes_{\mO_X}\Omega_{X/S}$,
which is compatible with the multiplication of $\mO_E$ and 
fits into the following commutative diagram
$$\xymatrix{\mO_E\ar[d]^{\nabla}\ar[r]^a&\mO_{E}\otimes\mO_{\mG}
\ar[d]^{\nabla\otimes 1+1\otimes\nabla_{\mG}}
\\\mO_E\otimes\Omega_{X/S}\ar[r]^{a}&
(\mO_E\otimes\mO_{\mG})\otimes\Omega_{X/S}}$$ where
$a:\mO_E\ra\mO_E\otimes\mO_{\mG}$ is the co-action map. Equivalently, a flat connection on $E$ is an isomorphism $\nabla:p_1^*(E)\is p_2^*(E)$ of $p_1^*\mG\is p_2^*\mG$-torsors on $\Delta^{(1)}$ that restricts to the identity
on $\Delta$ and that satisfies the usual cocycle condition on $\Delta^{(1)}_3$.

We denote by $\Loc_{\mG}$ the stack of $\mG$-torsors with flat
connections (or $\mG$-local systems). Note that the following
discussions do not require the representability of this stack.

\begin{example}
In the case of constant group scheme $\mG=G\times_S X$, there is a
canonical connection on $\mO_{\mG}=\mO_G\otimes_{\mO_S} \mO_X$ coming from
$\mO_X$. The above definition then reduces to the standard one.
\end{example}

\subsection{Lie algebroid definition}\label{Lie algebroid}
Here is an equivalence definition. Let $\mG$ be a smooth affine
$\mD_{X/S}$-group scheme as before. Let $E$ be a $\mG$-torsor. Let
us denote by $\widetilde{T}_E$ the Lie algebroid of infinitesimal
symmetry of $E$: a section of $\widetilde{T}_E$ is a pair $(v,\tilde
v)$, where $v\in T_{X/S}$ and $\tilde v\in T_{E}$ is a vector field
on $E$ such that:
\begin{enumerate}
\item The restriction on $\tilde v$ to $\mO_X\subset \mO_E$ is equal to
$v$ (i.e. $\tilde v$ is a lifting of $v$).

\item $\tilde v$ is $\mG$-invariant, i.e. the following diagram
commutes
$$\xymatrix{\mO_E\ar[r]^{a\ \ \ }\ar[d]^{\tilde v}&\mO_E\otimes\mO_{\mG}\ar[d]^{\tilde v
\otimes 1+ 1\otimes\nabla_{\mG}(v)}
\\\mO_E\ar[r]^{a\ \ \ }&\mO_E\otimes\mO_{\mG}}$$
\end{enumerate}

Let $\sigma:\widetilde{T}_{E}\ra T_{X/S}$ be the projection map
$(v,\tilde v)\ra v$. There is the following exact sequence
\begin{equation}\label{Atiyah algebroid}
0\to \ad(E)\to \widetilde{T}_E\stackrel{\sigma}{\to} T_{X/S}.
\end{equation}

A connection $\nabla$ on $E$ is a splitting of this exact
sequence, i.e. $\nabla$ is a map $\nabla:T_{X/S}\ra \widetilde{T}_E$
such that $\sigma\circ\nabla=\id$. If in addition $\nabla$ is a Lie
algebroid homomorphism, we say that $\nabla$ is a flat connection.


\subsection{Connections on the trivial $\mG$-torsor}\label{conn on triv}
Let $E=E^0$ be the trivial $\mG$-torsor. Then it is equipped with a canonical flat connection coming from $\nabla_\mG$, denoted by $\nabla^0$. Then by \S \ref{Lie algebroid}, the sheaf of connections on $E^0$ is isomorphic to $\Lie \mG\otimes\Omega_{X/S}$ via $\nabla\mapsto \nabla-\nabla^0$. We denote the subsheaf of flat connections on $E^0$ by $(\Lie \mG\otimes\Omega_{X/S})^{cl}$. If $\mG=G\times X$ is constant, then $(\Lie \mG\otimes\Omega_{X/S})^{cl}=\Lie G\otimes \calZ_{X/S}$, where $\calZ_{X/S}\subset\Omega_{X/S}$ is the sheaf of closed one-forms.

There is always the following map of
sheaves on $X$
\[d\log: \mG\to (\Lie \mG\otimes \Omega_{X/S})^{cl},\]
defined as follows. We regard $g\in\calG$ as an element in $\calH
om_{\mO_X}(\calO_\mG,\mO_X)$, which carries on a natural connection,
still denoted by $\nabla_\mG$. Let $\calI=\ker g\subset\mO_\mG$.
Then it is easy to see that $\nabla_\mG (g)\in \calH
om(\calO_\mG,\mO)\otimes\Omega_{X/S}$ annihilates $\mI_g^2$, and therefore
induces a map
\[\nabla_\mG (g): \mI_g/\mI_g^2\to \Omega_{X/S}.\]
Then $g^{-1}\nabla_\mG(g)$ can be regarded as an
element in $\Lie \mG\otimes\Omega_{X/S}$, and $\nabla^0+ d\log(g)$ defines a connection on $E^0$.
Note that $\mG$ acts on $E^0$, and therefore on $\tilde{T}_{E^0}$. It is easy to see that $g: (E^0, \nabla^0+d\log(g))\simeq (E^0,\nabla^0)$ is an isomorphism. In particular, $\nabla^0+d\log(g)$ is a flat connection, i.e. $d\log(g)\in (\Lie\mG\otimes\Omega_{X/S})^{cl}$.

\subsection{Bitorsors and connection on bitorsors}\label{bitorsor}
Let $\mG_1$ and $\mG_2$ be two smooth affine  group schemes. A
$(\mG_1\times\mG_2)$-bitorsor  on $X$ is a scheme $E$ on $X$ with a
$(\mG_1\times\mG_2)$-action that makes $E$ into a left
$\mG_1$-torsor and a right $\mG_2$-torsor. If $\mG_1$ and $\mG_2$ are $\mD_{X/S}$ group schemes, we can similarly define the
notion of a flat connection on a $(\mG_1\times\mG_2)$-bitorsor $E$,
i.e. an isomorphism $\nabla:p_1^*(E)\is
p_2^*(E)$ of bitorsors satisfying the usual conditions as before.

We denote by $\Loc_{\mG_1\times\mG_2}$ the stack of
$(\mG_1\times\mG_2)$-bitorsors with flat connections.

\begin{example}\label{autE}
Let $\mG$ be a smooth affine $\mD_{X/S}$-group scheme over $X$ and let
$E\in\Loc_{\mG}$. Let $\Aut(E)$ be the group scheme of automorphisms
of $E$ (as a $\mG$-torsor). Then $E$ has a natural
$(\Aut(E)\times\mG)$-bitorsor structure. In addition, the group
scheme $\Aut(E)$ has a canonical flat connection and
$E\in\Loc_{\Aut(E)\times\mG}$. To see this, observe that
$\Aut(E)\simeq E\times^{\mG}\mG$ and for any $Y\ra X$ we have
$\Aut(E)_Y\is E_Y\times^{\mG_Y}\mG_Y$. Thus, the connections on $E$
and on $\mG$ induce an isomorphism $p_1^*\Aut(E)\is p^*_2\Aut(E)$ on $\Delta^{(1)}$, which
defines a connection on $\Aut(E)$. It is
clear from the above construction that $E\in\Loc_{\Aut(E)\times\mG}$.
\end{example}

\subsection{Some functors}\label{some functors}
Let $f:Y\ra X$ be morphism between smooth schemes and let $\mG_X$ be
a smooth affine $\mD_{X/S}$-group scheme on $X$. Let
$\mG_Y:=f^*\mG_X$. Then $\mG_Y$ is a $\mD_{Y/S}$-group scheme and the pullback of $f$ defines a functor
\begin{equation}\label{pullback}
f^*:\Loc_{\mG_X}\ra\Loc_{\mG_Y},
\end{equation}
sometimes called the pullback functor.

Assume that $\mG_1$ and $\mG_2$ are two smooth affine $\mD_{X/S}$-group schemes  and $h:\mG_1\ra\mG_2$ is a horizontal group scheme homomorphism. For any
$(E,\nabla)\in\Loc_{\mG_1}$ we can form the usual induced $\mG_2$-torsor
$h_*E=E\times^{\mG_1,h}\mG_2$.  Moreover, there is a canonical isomorphism $p_1^*h_*E\simeq p_2^*h_*E$ on $\Delta^{(1)}$
which defines a
connection $h_*\nabla$ on $h_*E$. Thus the assignment $(E,\nabla)\ra
(h_*E,h_*\nabla)$ defines a functor
\begin{equation}\label{ind}
h_*:\Loc_{\mG_1}\ra\Loc_{\mG_2},
\end{equation}
sometimes called the induction functor.

Let $\mG_1$ and $\mG_2$ be two $\mD_{X/S}$-groups. Let
$P\in\Loc_{\mG_1}$ and $E\in\Loc_{\mG_1\times\mG_2}$. By a similar construction as above, we can form
the induced $\mG_2$-torsor
$$P\otimes E:= P\times^{\mG_1}E,$$
equipped with a flat connection $\nabla_{P\otimes E}$. The
assignment $(P,\nabla_P)\times (E,\nabla_E)\ra (P\otimes
E,\nabla_{P\otimes E})$ defines a functor
\begin{equation}\label{tensor}
\otimes:\Loc_{\mG_1}\times \Loc_{\mG_1\times\mG_2}\ra\Loc_{\mG_2},
\end{equation}
sometimes called the tensor functor.

\subsection{The $p$-curvature}\label{appen:p-curv}
Let us now assume that $p\calO_S=0$. Let us first recall the notion of the $p$-curvature of a $\mD_{X/S}$-module.

Let $(\mM,\nabla:\mM\ra\mM\otimes\Omega_{X/S})$ be a $\mD_{X/S}$-module. For any $v\in T_{X/S}$, let $v^{[p]}\in T_{X/S}$ be the $p$-th power of $v$.
By \cite[Proposition 5.2]{K},
the map 
\[T_{X/S}\ra (F_X)_*\End_{\mO_X}(\mM),\ \ v\ra\nabla(v)^p-\nabla(v^{[p]})\]
is a morphism of 
$\mO_X$-modules. 
By adjunction, we 
get a map 
\[\Psi(\nabla):F^*T_{X'/S}\ra\End_{\mO_X}(\mM),\] 
called the $p$-curvature of $(\mM,\nabla)$. 

Let $N\to X$ be an affine $\mD_{X/S}$-scheme, we define the $p$-curvature of $N$ as the $p$-curvature $\Psi(\nabla_N)$ of the $\mD_{X/S}$-algebra $(\mO_N,\nabla_N)$. Using the formula of the $p$-curvature of a tensor connection $\Psi(\nabla_1\otimes\nabla_2)=\Psi(\nabla_1)\otimes\id+\id\otimes\Psi(\nabla_2)$, we see that $\Psi(\nabla_N)$ factors through
\[\Psi(\nabla_N): F^*T_{X'/S}\to \Der_{\mO_X}(\mO_N)\subset
\End_{\mO_X}(\mO_N).\]

Now let $\mG$ be a smooth affine $\mD_{X/S}$-group scheme. It is easy to see that for any $v\in F^*T_{X'/S}$, the following diagram is
commutative
\[\begin{CD}
\mO_\mG@>m>>\mO_\mG\otimes_{\mO_X}\mO_{\mG}\\
@V\Psi(\nabla_\mG)(v) VV@VV\Psi(\nabla_\mG)(v)\otimes \id+\id\otimes \Psi(\nabla_\mG)(v) V\\
\mO_{\mG}@>m>>\mO_{\mG}\otimes_{\mO_X}\mO_{\mG}.
\end{CD}\]
Likewise, let $(E,\nabla)\in\Loc_{\mG}$. Then for any $v\in F^*T_{X'}$, the following diagram is
commutative
\begin{equation}\label{p-curv for}
\begin{CD}
\mO_E@>m>>\mO_E\otimes_{\mO_X}\mO_{\mG}\\
@V\Psi(\nabla)(v) VV@VV\Psi(\nabla)(v)\otimes \id+\id\otimes \Psi(\nabla_\mG)(v) V\\
\mO_E@>m>>\mO_E\otimes_{\mO_X}\mO_{\mG}.
\end{CD}
\end{equation}
In particular, if $\Psi(\nabla_\mG)=0$, then the image of $\Psi(\nabla)$ lands in
$\ad(E)\subset \Der_{\mO_X}(\mO_E)$, and therefore the $p$-curvature mapping can be regarded as a
section
\[\Psi(\nabla)\in \ad(E)\otimes F^*\Omega_{X'/S},\]
which reduces to the standard notion of the $p$-curvature of a principal bundle with a flat connection (e.g. see \cite[Appendix]{B}).

\subsection{The scheme of horizontal sections}\label{horizontal sect}
To continue, we develop the theory of schemes of horizontal sections of a $\mD_{X/S}$-scheme in characteristic $p$, analogous to \cite[Proposition 2.6.2]{BD}.
\begin{lemma}[Cartier descent]\label{Cd}
\
\begin{enumerate}
\item
For any quasi-coherent sheaf $\mL$ on $X'$, there is a canonical
$\mD_X$-module structure on $F^*(\mL)$ and the assignment $\mL\ra
F^*(\mL)$ defines an equivalence between the category of
quasi-coherent sheaves on $X'$ and the category of $\mD_X$-modules
on $X$ with zero p-curvature, with an inverse functor given
by taking flat sections $\mF\mapsto \mF^\nabla$.

\item The above equivalence is a tensor equivalence, i.e.,
for any $\mL_1,\mL_2\in\on{QCoh}(X')$ the natural isomorphism of $\mO_{X'}$-modules
\[m:F^*(\mL)\otimes F^*(\mL_2)\is F^*(\mL_1\otimes\mL_2)\]
is compatible with their $\mD_X$-modules
structures coming from part (1).

\item Let $\mB$ be a $\mO_{X'}$-algebra on $X'$. Then there is a
canonical $\mD_X$-algebra structure on $F^*(\mB)$ and the assignment
$\mL\ra F^*(\mL)$ defines an equivalence between the category of
$\mO_{X'}$-algebras on $X'$ and the category of $\mD_X$-algebra on
$X$ with zero p-curvature.
\end{enumerate}
\end{lemma}
\begin{proof}
Part (1) is the standard Cartier descent. Part (3) follows from Part
(2), which can be proved by a direct computation.
\end{proof}

\begin{prop}
\
\begin{enumerate}
\item
The functor $\mM'\ra F^*\mM'$ admits a left adjoint functor, i.e.
for a $\mD_X$-algebra $\mN$ there is a $\mO_{X'}$-algebra
$H_{\nabla}(\mN)$ such that
$$\on{Hom}_{\mD_X\on{-alg}}(\mN,F^*\mM')=\on{Hom}_{\mO_{X'}\on{-alg}}(H_{\nabla}(\mN),\mM')$$
for any $\mO_{X'}$-algebra $\mM'$.
\item
The canonical map $\mN\ra F^*(H_{\nabla}(\mN))$ is surjective.
\end{enumerate}
\end{prop}

\begin{proof}
Let $\nabla$ be the corresponding connection of $\mN$. We denote by $\Psi$ the
$p$-curvature of $\nabla$. We can think of $\Psi$ as a map
$\Psi:F^*T_{X'}\otimes_{\mO_X}\mN\ra\mN$. Let $\Psi(\mN)$ be the
ideal of $\mN$ generated by the image of $\Psi$ and define
$\mN_{\Psi}=\mN/\Psi(\mN)$. Since $\nabla$ commutes with $\Psi$, the
$\mO_X$-algebra $\mN_{\Psi}$ carries a connection and we define the
following $\mO_{X'}$-algebra
$$H_{\nabla}(\mN)=(\mN_{\Psi})^{\nabla}.$$
Let us show that $H_{\nabla}(\mN)$ satisfies our requirement. For
any $\mO_{X'}$-algebra $\mM'$, the $p$-curvature of $F^{*}\mM'$ is
zero, and therefore we have $$\on{
Hom}_{\mD_X\on{-alg}}(\mN,F^{*}\mM')=\on{Hom}_{\mD_X\on{-alg}}(\mN_{\Psi},F^*\mM')$$
On the other hand, since the $p$-curvature of $\mN_{\Psi}$ is zero,
by Lemma \ref{Cd}, we have
$$\on
{Hom}_{\mD_X\on{-alg}}(\mN_{\Psi},F^*\mM')=\on{Hom}_{\mO_{X'}\on{-alg}}((\mN_{\Psi})^{\nabla},\mM')
=\on{Hom}_{\mO_{X'}\on{-alg}}(H_{\nabla}(\mN),\mM').$$ Therefore
$$\on{
Hom}_{\mD_X\on{-alg}}(\mN,F^{*}\mM')=\on{Hom}_{\mO_{X'}\on{-alg}}(H_{\nabla}(\mN),\mM').$$
This proved part 1).

By Lemma \ref{Cd} again, we have $\mN_{\Psi}=F^*(H_{\nabla}(\mN))$
and it implies the canonical map
$\mN\ra\mN_{\psi}=F^*(H_{\nabla}(\mN))$ is surjective. This proved
part 2).
\end{proof}

Let $\mN$ be a commutative $\mD_X$-algebra. Then the
$\mO_{X'}$-algebra $H_{\nabla}(\mN)$ is commutative by $(2)$. We
called the $X'$-scheme $$N^{\nabla}:=\on{Spec}(H_{\nabla}(\mN))$$
the scheme of horizontal sections. Let $N$ and $N^{\Psi}$ be the
$\mD_X$-scheme associated to the commutative $\mD_X$-algebras $\mN$
and $\mN_{\Psi}$.  We have
\begin{equation}\label{horizontal}
N^{\Psi}=X\times_{X'}N^{\nabla}, \quad (N^{\Psi})^\nabla=N^\nabla,
\end{equation}
and $N^{\Psi}\subset N$ is the maximal closed subscheme of $N$ that
is constant with respect to the connection.

\begin{example}\label{Cartier}
Let $\mG'$
be a smooth affine group scheme over $X'$. Then by the theory of Cartier decent, the group scheme
$\mG:=F^*\mG'$ has a canonical connection, and for any $\mG'$-torsor
$E'$, there is a canonical connection on the $F^{*}\mG'$-torsor
$F^{*}E'$. One can easily check that the
$p$-curvature of this connection is zero. Therefore, we have a
functor
$$F^*:\Bun_{\mG'}\ra\Loc_{\mG}, \quad E'\ra F^*E'.$$
The functor $F^*$ induces an equivalence between the category of
$\mG'$-torsors on $X'$ and the category of $\mG$-local systems with
zero $p$-curvature. The inverse is given by $E\mapsto E^\nabla$.
\end{example}

\subsection{The commutative case}\label{The commutative case}
Assume that $\mG'$ is commutative on $X'$ and $\mG=F^*\mG'$. For any
$\mG$-torsor $E$, $\ad(E)\simeq \on{Lie}\mG$ canonically. If
$(E,\nabla)\in\Loc_{\mG}$, then by the Cartier descent,
$\Psi(\nabla)\in \Gamma(X,\on{Lie}\mG\otimes F^*\Omega_{X'/S})$ is
the pullback of a unique element in $\Gamma(X',\Lie\mG'\otimes
\Omega_{X'/S})$. Therefore, taking $p$-curvature can be regarded as
a map
\[h_p:\Loc_{\mG}\to \Gamma(X',\Lie\mG'\otimes \Omega_{X'/S}),\]
where $T=\Gamma(X',\Lie\mG'\otimes \Omega_{X'/S})$ is regarded as a space over $S$, whose fiber over $s\in S$ is $\Gamma(X'_s,\Lie\mG'|_{X'_s}\otimes\Omega_{X'_s})$.
This is called the $p$-Hitchin map (see Theorem \ref{p-hitchin} for the
non-commutative analogue). 

As $\mG$ is commutative, given
two $\mG$-local systems $(E_1,\nabla_1),(E_2,\nabla_2)$, the
induction $E_1\times^\mG E_2$ is a natural $\mG$-local system. Then
we have
\begin{equation}\label{appen:add}
h_p(E_1\times^\mG E_2)=h_p(E_1)+h_p(E_2).
\end{equation}
Combining with Example \ref{Cartier}, we have
\begin{corollary}\label{commutative Cartier}
Assume that $\mG'$ is commutative. Then the stack of $\mG$-local
systems with a fixed $p$-curvature $\psi\in T$ is a (pseudo) torsor under the
Picard stack $\Bun_{\mG'}$.
\end{corollary}

We will give an interpretation of this (pseudo) $\Bun_{\mG'}$-torsor, generalizing \cite[Proposition 4.13]{OV} for $\mG=\bbG_m$. 

Let $\calZ_{X/S}$ be the
sheaf of closed one forms on $X$. Then $F_*\calZ_{X/S}$ is an $\mO_{X'}$-module. 
Recall the definition of the sheaf of flat connections $(\Lie\mG\otimes\Omega_{X/S})^{cl}$ on the trivial $\mG$-torsor $E^0$ as in \S \ref{conn on triv}. Under our assumptions of $\mG$,
 $F_*(\Lie\mG\otimes\Omega_{X/S})^{cl}=\Lie\mG'\otimes_{\mO_{X'}} F_*\calZ_{X/S}$, where $\calZ_{X/S}$ is the sheaf of closed one-forms on $X$. 
 Then the
$p$-Hitchin map $h_p$ induces an additive map of  \'{e}tale sheaves on $X'$, 
\[h_p:\Lie\mG'\otimes F_*\calZ_{X/S}\to \Lie\mG'\otimes\Omega_{X'/S}.\]
As in the case $\mG=\bG_m$, this map fits into 
the following four term exact sequence of sheaves on $X'_{et}$.
\begin{prop}\label{appen:4-term}
Over $X'_{et}$, there is the exact sequence of  \'{e}tale sheaves
\[1\longrightarrow\mG'\longrightarrow F_*\mG\stackrel{F_*d\log}{\longrightarrow} \Lie\mG'\otimes F_*\calZ_{X/S}\stackrel{h_p}{\longrightarrow} \Lie\mG'\otimes \Omega_{X'/S}\longrightarrow 1,\]
where $d\log$ is defined as in \S \ref{conn on triv}.
\end{prop}
\begin{proof}
It is clear from the definition of $d\log$ that the kernel of $F_*d\log$ is $\mG^\nabla=\mG'$. Next, we show that the sequence is exact at $\Lie\mG'\otimes F_*\calZ_{X/S}$. As explained in \S \ref{conn on triv}, $(E^0,\nabla^0+d\log(g))\simeq (E^0,\nabla^0)$. As the latter has zero $p$-curvature, we see that $h_p\circ F_*d\log=0$. On the other hand, if $\nabla=\nabla^0+\omega$ is a flat connection on $E^0$ with zero $p$-curvature, then by Example \ref{Cartier} $E'=(E^0)^\nabla$ is a $\mG'$-torsor. and there is a canonical isomorphism $\alpha:F^*E'\simeq E^0$. \'{E}tale locally, we can trivialize $E'$ so we can choose $\beta:(E^0)'\simeq E'$. Then we obtain a section $g=\alpha\circ F^*\beta\in \mG$, and it is not hard to check that $\omega=d\log(g)$. Note that if $\mG=\bbG_m$, the argument shows that this sequence is exact at $\Lie\mG'\otimes F_*\calZ_{X/S}$ even Zariski locally on $X'$, which is well-known.

Finally, we show that $h_p$ is surjective. Unlike $\mG=\bG_m$, there is no explicit formula for the $p$-linear map of $\Lie\mG$, and the usual argument by explicit calculations does not apply here.

Recall that if $\calV$ is a locally free $\mO_{X'}$-module of finite rank, $\Spec\on{Sym}_{\mO_{X'}}\calV^\vee$ is a vector bundle on $X'$ whose sheaf of sections are $\calV$.
Let $\bbV_1$ and $\bbV_2$ be the vector bundles on $X'$ corresponding to $\Lie\mG'\otimes F_*\calZ_{X/S}$ and $\Lie\mG'\otimes\Omega_{X'/S}$ respectively. Although $h_p$ is not $\mO_{X'}$-linear, we have
\begin{lem}The map $h_p$ is induced by a smooth surjective homomorphism of commutative group schemes $\mathbbm{h}_p:\bbV_1\to \bbV_2$.
\end{lem}
Assuming the lemma, the surjectivity of $h_p$ then is clear. Namely, let $s: X'\to \bbV_2$ be a section, then \'{e}tale locally on $X'$, $s$ can be lifted to a section of $\bbV_1$ as $\mathbbm{h}_p$ is smooth surjective. 

It remains to prove the lemma. 
We use an argument similar to \cite[Proposition 2.5 (1)]{OV}.
We first give another description of the map $h_p$. Let $\mathcal C_{X/S}:  F_*\calZ_{X/S}\to \Omega_{X'/S}$ be
the Cartier operator. Let $\pi_{X/S}:X'\to X$ is the map over the
absolute Frobenius $Fr_S$ of $S$. Let $(-)^{p}:\Lie \mG'\to \Lie\mG'$ be the map given by the $p$-Lie algebra structure of $\Lie\mG'$. We claim that for $z\otimes \omega\in \Lie \mG'\otimes F_*\calZ_{X/S}$, 
\begin{equation}\label{explicit hp}
h_p(z\otimes\omega)= z^p\otimes \pi^*_{X/S}(\omega)- z\otimes \calC_{X/S}(\omega).
\end{equation}
Indeed, let $\nabla^0$ be the canonical connection on the trivial $\mG$-torsor $E^0$. 
Then for a vector field $\xi\in T_{X/S}$, and $z\otimes \omega\in \Lie\mG'\otimes F_*\calZ_{X/S}$, one has (as sections of $\widetilde{T}_{E^0}$)
\[(\nabla^0_\xi+(\omega,\xi)z)^p-((\nabla^0_\xi)^p+(\omega,\xi^p)z)=\xi^{p-1}(\omega,\xi)z+(\omega,\xi)^pz^p-(\omega,\xi^p)z,\]
by the Jacobson identity. By the definition of the Cartier operator, $(\omega,\xi^p)-\xi^{p-1}(\omega,\xi)=(\calC_{X/S}(\omega), \pi^*_{X/S}(\xi))$. The claim follows. 

We apply \eqref{explicit hp} as follows:
The Cartier operator $\calC_{X/S}$ is $\calO_{X'}$-linear and surjective, and therefore is induced by a smooth homomorphism $\bbC_{X/S}: \bbV_1\to \bbV_2$ of the underlying commutative group schemes. The map $z\otimes \omega\mapsto z^p\otimes \pi^*_{X/S}(\omega)$ is not $\mO_{X'}$-linear, but is still induced from an inseparable homomorphism $\Phi:\bbV_1\to \bbV_2$ of the commutative group schemes. In fact, let $Fr_{X'}:X'\to X'$ be the absolute Frobenius of $X'$. We have $\mO_{X'}$-linear maps $Fr_{X'}^*\Lie\mG'\to \Lie\mG'$ given by the $p$-Lie algebra structure and  $Fr_{X'}^*F_*\calZ_{X/S}\to \Omega_{X'/S}$ given by the composition $Fr_{X'}^*F_*\calZ_{X/S}\to Fr_{X'}^*F_*\Omega_{X/S}\to \Omega_{X'/S}$. Therefore, we have a commutative group scheme homomorphism $\bbV_1^{(X')}\to \bbV_2$. It is then readily to see that $\Phi$ is the composition of the relative Frobenius $\bbV_1\to \bbV_1^{(X')}$ (which is a group homomorphism) with the above homomorphism.
 
Now $\mathbbm{h}_p=\Phi-\bbC_{X/S}$ is a smooth morphism, as $\Phi$ is inseparable and the differential of $\bbC_{X/S}$ is surjective. It remains to show that $\Phi-\bbC_{X/S}$ is surjective at the level of points. We can base change $\mathbbm{h}_p$ to a geometric point of $x\in X'$. As $\mathbbm{h}_p$ is smooth, the image contains an open subgroup of $(\bbV_2)_{x}$, as $(\bbV_2)_{x}$ is connected, this subgroup must be the entire $(\bbV_2)_{x}$. The lemma follows.
\end{proof}

Now, we will give an interpretation of the pseudo $\Bun_{\mG'}$-torsor in Corollary \ref{commutative Cartier}.  Namely, the $p$-curvature $\psi$ gives rise to a section of $\Lie\mG'\otimes \Omega_{X'/S}$, and the 4-term exact sequence \ref{appen:4-term} induces a $\mG'$-gerbe $\sG_\psi$ on $X'$. 

\begin{prop}\label{appen: equiv two gerbes}
Then the pseudo $\Bun_{\mG'}$-torsor as in Corollary \ref{commutative Cartier} is the stack of splittings of the gerbe $\sG_\psi$ over $S$.
\end{prop}

The proof is similar to \cite[\S 4, Proposition 4.2]{OV}. One only needs to replace the invertible sheaf $L$ associated to the $F_{Y/S*}(\mO_Y^*)$-torsor $\mL$ in \emph{loc. cit.} by the structure sheaf $\mO_E$ of the $F_*\mG$-torsor $E$ in the current setting. 

\subsection{}\label{tangent map} 
In this subsection we assume that $S=\Spec k$, where $k$ algebraically closed of characteristic $p>0$. We describe the tangent map of $h_p:\Loc_\mG\ra\Gamma(X',\Lie\mG'\otimes\Omega_{X'/k})$ at the trivial $\mG$-torsor with the canonical connection $x=(E^0,\nabla^0)$. First, the tangent space of $\Loc_\mG$ at $x$ is given by hypercohomology $\mathbb H^1$ of 
the following deRham complex
\[\Omega_{X/k}^\bullet(\Lie\mG):=\{0\ra\Lie\mG\stackrel{\nabla^0}{\ra}\Lie\mG\otimes\Omega_{X/k}^1\ra\Lie\mG\otimes\Omega_{X/k}^2\to\cdots\}.\] 
\quash{and the multiplicative $\Lie\mG$-value de Rham complex
\[\Omega_{X/S}^\bullet(\mG):=0\ra\mG\stackrel{dlog}\ra\Lie\mG\otimes\Omega^1_{X/S}\ra\Lie\mG\otimes\Omega_{X/S}^2\ra\cdot\cdot\cdot\]
where $dlog:\mG\ra\Lie\mG\otimes\Omega^1_{X/S}$ is the map in \ref{conn on triv}. 

By a similar argument in \cite[Section 4]{MM} using $\Omega_{X/S}^\bullet(\Lie\mG)$ and $\Omega_{X/S}^\bullet(\mG)$
one can show there is an isomorphism 
\[\Lie(\Loc_\mG)\is\mathbb H^1(X,\Omega_{X/S}^\bullet(\Lie\mG)).\]
The conjugate filtration on the de Rham complex $\Omega_{X/S}^\bullet$} 
Recall that there is the ``second spectral sequence"
with $E^i=\mathbb H^i(X,\Omega_{X/k}^\bullet(\Lie\mG))$
and 
$E_2^{i,j}=H^i(X,\mathcal H^j(\Omega_{X/k}^\bullet(\Lie\mG)))$. When $\mG=\bG_m$, this is also known as the conjugate spectral sequence.
In particular
\begin{equation}\label{appen:conjE2}
E_2^{0,1}=\Gamma(X,\mathcal H^1(\Omega_{X/k}^\bullet(\Lie\mG))\simeq \Lie\mG'\otimes\Omega_{X'/k}.
\end{equation}
Here the last isomorphism is obtained by the Lie algebra version of the exact sequence as in Proposition \ref{appen:4-term}. Note that if $\mG=\bG_m$ or $\bG_a$, it is just the usual Cartier isomorphism. The edge morphism of the spectral sequence induces
\[
c:T_x\Loc_\mG\is\mathbb H^1(X,\Omega_{X/k}^\bullet(\Lie\mG))\ra H^0(X',\Lie\mG'\otimes\Omega_{X'/k}^1).
\]

\begin{lemma}\label{appen:tangent}
The map $c$ is equal to  $-dh_p$.
\end{lemma}
\begin{proof}
We follow the argument in \cite[Lemma 4.12]{OV}. Recall that we have a natural inclusion $i:H^0(X',\Lie\mG'\otimes F_*\mathcal Z_{X/k}^1)\ra\Loc_\mG$ of the substack of flat connections on the 
trivial torsor $E^0$. It follows by definition that
\[-d(h_p\circ i)=c\circ di:
H^0(X,\Lie\mG'\otimes F_*\mathcal Z_{X/k}^1)\to H^0(X',\Lie\mG'\otimes\Omega_{X'/k}^1).\]
More explicitly, they are equal to the Cartier map $\mathcal C_{X/S}$ as introduced in
the proof of Proposition \ref{appen:4-term}. Now
let $v\in\mathbb H^1(X,\Omega_{X/k}^\bullet(\Lie\mG))$
and let $(E,\nabla)$ be the corresponding $\mG$-torsor with connection on $X[\epsilon]$.
The section $dh_p(v)\in H^0(X',\Lie\mG'\otimes\Omega_{X'/k}^1)$ is determined by its pull back to
any \'etale cover of $X'$. 
We can choose an \'etale cover on which $E$ is trivial, hence reduced to the case   
when $E$ is trivial and $v\in H^0(X,\Lie\mG'\otimes F_*\mathcal Z_{X/k}^1)$.
\end{proof}
\quash{\begin{corollary}
If $X\ra S$ has relative dimension less or equal to one, then $h_p:\Loc_\mG\ra H^0(X',\Lie\mG'\otimes\Omega_{X'/S}^1)$
is a smooth map.
\end{corollary}
\begin{proof} 
By assumption we have $E_2^{2,0}=H^2(X',\Lie\mG')=0$. Thus above lemma 
implies $dh_p$ is surjective, and it follows that $h_p$ is smooth.
\end{proof}}




\begin{thebibliography}{99}


\bibitem[BNR]{BNR}A. Beauville, M.S. Narasimhan, S. Ramanan: {\it Spectral curves and the generalised theta divisor}, J. Reine Angew. Math. (1989)
Volume: 398, page 169-179.



\bibitem[BD]{BD} A. Beilinson, V. Drinfeld: {\it Quantization of Hitchin's integrable
system and Hecke eigensheaves}, preprint, available at
http://www.math.uchicago.edu/~mitya/langlands/.

\bibitem[BB]{BB} R. Bezrukavnikov. A. Braverman.: {\it Geometric Langlands
conjecture in characteristic p: The $GL_n$ case },
Pure Appl. Math. Q. 3 (2007), no. 1, Special Issue: In honor of Robert D. MacPherson. Part 3, 153-179.

\bibitem[B]{B} J.B. Bost: {\it Algebraic leaves of algebraic foliations over number fields},
Publ. Math. Inst. Hautes \'Etud. Sci. No. 93 (2001), 161-221.


\bibitem[CGP]{CGP} B. Conrad, O. Gabber, G. Prasad: {\it Pseudo-reductive groups}, 
New mathematical momographs: 17.


\bibitem[CZ]{CZ}T.H. Chen, X. Zhu: {\it Geometric Langlands in prime characteristic}, arXiv:1403.3981.

\bibitem[DG]{DG} R. Donagi, D. Gaitsgory: {\it The gerbe of Higgs bundles},
Transformation groups Volume 7, Number 2, 109-153

\bibitem[G]{G}  M. Groechenig: {\it Moduli of flat connections in positive characteristic},
arXiv:1201.0741.

\bibitem[I]{I}L. Illusie: {\it Complexe de deRham-Witt et cohomologie cristalline},
Ann. Sci. \'Ecole Norm. Sup. (4) 12 (1979), no. 4, 501--661. 

\bibitem[K]{K}N. Katz: {\it Nilpotent connections and the Monodromy Theorem},
Publ. Math. Inst. Hautes \'Etud. Sci. No. 39 (1970), 175-232.


\bibitem[LP]{LP}Y. Laszlo, C. Pauly: {\it On the Hitchin morphism in
positive characteristic}, Internat. Math. Res. Notices 2001, no. 3,
129--143.

\quash{\bibitem[M]{M} S. Mochizuki: {\it A theory of ordinary p-adic curves}, 
Publ. RIMS Kyoto Univ.  32 (1996), 957--1151.}

\bibitem[MM]{MM} B. Mazur, W. Messing: {\it Universal extensions and one dimensional crystalline cohomology}, Lecture Notes in Mathematics 370.


\bibitem[N1]{N1} B.C. Ng$\hat{\on o}$: {\it Hitchin fibration and endoscopy },
Invent. Math. 164 (2006) 399-453.

\bibitem[N2]{N2} B.C. Ng$\hat{\on o}$: {\it Le lemme fondamental pour les alg\`ebres
de Lie}, Publ. Math. Inst. Hautes \'Etud. Sci. No. 111 (2010), 1-169.

\bibitem[O]{O}A. Ogus: {\it Higgs cohomology, $p$-curvature, and the Cartier isomorphism}, Compositio Math. 140 (2004) 145--164.

\bibitem[OV]{OV}A. Ogus, V. Vologodsky: {\it Nonabelian Hodge theory in
characteristic p}, Publ. Math. Inst. Hautes \'Etud. Sci. No. 106
(2007), pp.1--138

\bibitem[S]{S}C. Simpson: {\it Higgs bundles and local systems}, 
Publ. Math. Inst. Hautes \'Etud. Sci. No. 75
(1992), 5--95.
\end{thebibliography}
\end{document}